\documentclass[12pt]{amsart}
\usepackage{epsfig,comment}
\usepackage{mathtools}
\usepackage[subnum]{cases}
\usepackage{enumerate}
\usepackage{bbm,bm}
\usepackage{amsmath,amssymb,mathrsfs}
\usepackage{cite}
\usepackage{mathtools}
\usepackage[font=small,labelfont=bf]{caption}
\usepackage{color}
\usepackage[bookmarks=true,
bookmarksnumbered=true, breaklinks=true,
pdfstartview=FitH, hyperfigures=false,
plainpages=false, naturalnames=true,
colorlinks=true,pagebackref=true,
pdfpagelabels,colorlinks=true,
	linkcolor=blue,
	citecolor=black,
	filecolor=black,
	urlcolor=black]{hyperref}
\usepackage{wasysym}
\usepackage{graphicx}

\usepackage{pgfplots}
\pgfplotsset{width=10cm,compat=1.9}
\usepgfplotslibrary{fillbetween}

\newcommand\xqed[1]{%
  \leavevmode\unskip\penalty9999 \hbox{}\nobreak\hfill
  \quad\hbox{#1}}
\newcommand\dssymb{\xqed{\small //}}

\newtheorem{theorem}{Theorem}[section]
\newtheorem{definition}[theorem]{Definition}
\newtheorem{lemma}[theorem]{Lemma}
\newtheorem{proposition}[theorem]{Proposition}
\newtheorem{corollary}[theorem]{Corollary}

\newtheorem{thmx}{Theorem}
\theoremstyle{remark}
\newtheorem{remark}[theorem]{Remark}

\newcommand{\N}{\mathbb{N}} 
\newcommand{\R}{\mathbb{R}}
\newcommand{\Rn}{\mathbb{R}^n}

\newcommand{\Sp}{\mathbb{S}^{n-1}}

\newcommand{\K}{\mathcal{K}}
\newcommand{\Kn}{\mathcal{K}^n}

\newcommand{\Ycoll}{\mathscr{Y}}
\newcommand{\SOn}{{\rm SO}(n)}
\newcommand{\SOo}{{\rm SO}(1)}
\newcommand{\On}{{\rm O}(n)}
\newcommand{\Oo}{{\rm O}(1)}
\newcommand{\dint}{\,\mathrm{d}}
\newcommand{\ind}{{\mathbf{I}}} 
\newcommand{\rot}{{\rm rot}}
\DeclareMathOperator{\conv}{conv}

\DeclareMathOperator{\dom}{dom}
\DeclareMathOperator{\epi}{epi}

\DeclareMathOperator{\vol}{vol}

\DeclareMathOperator{\supp}{supp}
\DeclareMathOperator{\pos}{pos}

\DeclareMathOperator*{\epilim}{epi-lim}
\DeclareMathOperator{\interior}{int}

\DeclareMathOperator{\infconv}{\mathbin{\Box}} 
\newcommand*\sq{\mathbin{\vcenter{\hbox{\rule{.4ex}{.4ex}}}}}

\newcommand{\Convc}{{\mbox{\rm Conv}_{{\rm c}}(\R^n)}}

\newcommand{\ConvcT}{{\mbox{\rm Conv}_{{\rm c}}(\R^2)}}
\newcommand{\LCc}{{\mbox{\rm LC}_{{\rm c}}(\R^n)}}

\newcommand{\Leg}{\mathcal{L}}

\DeclareMathOperator*{\bigsquare}{\scalerel*{\square}{\textstyle\sum}}

\usepackage{scalerel}

\newcommand\blfootnote[1]{%
  \begingroup
  \renewcommand\thefootnote{}\footnote{#1}%
  \addtocounter{footnote}{-1}%
  \endgroup
}

\title[Generalized outer linearizations and rotational epi-symmetrizations]{Generalized outer linearizations and extremal properties of rotational epi-symmetrizations}

\author{Steven Hoehner and Fabian Mussnig}
\date{}

\begin{document}

\begin{abstract}\noindent
We develop a functional extension of an extremal principle by Schneider\linebreak({\it Monatsh.\ Math.}, 1967) by introducing generalized outer linearizations of convex functions. Given a coercive convex function on $\Rn$, a generalized outer linearization is defined as a convex minorant represented by a general but function-dependent set of slopes, thereby extending classical outer representations of convex bodies by supporting halfspaces. This representation converts geometric outer approximations by supporting halfspaces into functional approximations by supporting affine functions, and replaces outer normal data by a dual sampling problem in the domain of the Legendre--Fenchel transform.

On a standard class of coercive convex functions, we derive a general extremal principle, showing that the rotational epi-symmetrization maximizes best approximations under outer linearizations of any monotone, concave functional that is upper semicontinuous with respect to epi-convergence. A central feature of the analysis is that it is carried out in the natural class of coercive, but not necessarily super-coercive, convex functions. Working in this setting introduces intricate topological and variational difficulties, which are addressed using refined duality and epi-convergence arguments.

As an application of our main results, we derive a functional version of Urysohn's inequality, as well as an analytic extension of a classical covering result of Firey and Groemer ({\it J. London Math. Soc.}, 1964). Finally, we prove an extremal inequality related to the piecewise affine approximation of convex functions.
\end{abstract}


\maketitle


\section {Introduction}
\subsection{An extremal principle for convex bodies}
One of the central themes in convex geometry is the representation and approximation of convex bodies by intersections of halfspaces with prescribed outer unit normals. A systematic study of such representations was initiated by Schneider \cite{Schneider-1967}, who obtained a far-reaching generalization of Urysohn's inequality. The latter asserts that among all convex bodies in $\Rn$ with fixed volume, the Euclidean ball uniquely maximizes mean width. More specifically, Urysohn's inequality states that for any \textit{convex body} $K$ in $\Rn$ (i.e.\ a non-empty, compact, convex subset),
\begin{equation}\label{urysohn-ineq}
\left(\frac{\vol_n(K)}{\vol_n(B_n)}\right)^{1/n} \leq \frac{w(K)}{w(B_n)}
\end{equation}
with equality if and only if $K$ is a Euclidean ball. Here and throughout the paper, $\vol_n$ denotes the $n$-dimensional volume (Lebesgue measure), $B_n$ and $\Sp$ are the Euclidean unit ball and unit sphere in $\Rn$, respectively, 
\begin{equation}
\label{eq:mean_width}
w(K)\coloneq\frac{2}{n\vol_n(B_n)}\int_{\Sp}h_K(u)\dint u
\end{equation}
is the \textit{mean width} of $K$, and $h_K(u)\coloneq\max_{x\in K}\langle x,u\rangle$ is the \textit{support function} of $K$, where $\langle \cdot,\cdot\rangle$ denotes the usual inner product on $\Rn$. Since the support function encodes the supporting halfspaces of a given convex body, Urysohn's inequality together with \eqref{eq:mean_width} clearly
highlights the role of support functions and families of outer unit normals in extremal problems.

\medskip

In its general form, Schneider's extremal principle from \cite{Schneider-1967}--formulated in Theorem~\ref{Schneider-main-result-A} and Theorem~\ref{Schneider-main-result-B} below--concerns convex bodies $K$ in $\Rn$ represented as intersections of halfspaces with outer normals belonging to a set $U\subseteq\Sp$. Schneider proved that, for a broad class of geometric functionals, the extremizers among all such representations exhibit strong structural properties, and that sharp inequalities can be derived without a priori finiteness assumptions on the set $U$. Classical results on circumscribed polytopes or circumscribed cylinders, for example, arise as special cases when $U$ is finite. This point of view suggests that Schneider’s theory in \cite{Schneider-1967} is not fundamentally about polytopes, but rather about outer representations of convex bodies controlled by sets of normal vectors, with finiteness playing the role of an additional constraint when one seeks quantitative approximation results.

\medskip

To formulate Schneider's result, let us first fix some of the notation that will be used throughout the paper. Let $\Kn$ denote the set of convex bodies in $n$-dimensional Euclidean space $\Rn$, where, for simplicity, we assume $n\geq 2$ throughout this introduction.
For a convex body  $K\in\Kn$ and a direction $u\in\Sp$, we denote by $H^-(K,u)$ the closed \textit{supporting halfspace} of $K$ with outer unit normal vector $u$, that is,
\[
H^-(K,u)\coloneq\{x\in\Rn: \, \langle x,u\rangle \leq h_K(u)\}.
\]
In addition, let
\[
\mathscr{U}_n\coloneq\{U\subseteq \Sp:\,\pos(U)=\R^n\},
\]
where $\pos(U)$ is the \textit{positive hull} of $U$.

Given $K\in\Kn$ and $U\in\mathscr{U}_n$, we consider the circumscribed set
\begin{equation}\label{schneider-def-1}
    P(K,U)\coloneq\bigcap_{u\in U}H^-(K,u).
\end{equation}
Since $\pos(U)=\Rn$, the set  $P(K,U)$ is bounded, and therefore it is a convex body that contains $K$. In particular, if $U\in\mathscr{U}_n$ is finite, then $P(K,U)$ is a convex polytope that circumscribes $K$, and the set of outer normal unit vectors of its facets (i.e., $(n-1)$-dimensional faces) is a subset of $U$.

Next, let $\SOn$ denote the rotation group of $\Rn$, and for $K\in\Kn$ and $\vartheta\in\SOn$, let $\vartheta K\in\Kn$ be the image of $K$ under the rotation $\vartheta$. Given a set $\mathfrak{U}\subseteq\mathscr{U}_n$ and a functional $\Phi\colon \Kn\to\R$, Schneider \cite{Schneider-1967} defined the quantity
\begin{equation}\label{schneider-eqn2}
    \Phi^*(K,\mathfrak{U})\coloneq\inf\{\Phi(P(\vartheta K,U)):\,\vartheta\in\SOn,\, U\in \mathfrak{U}\}
\end{equation}
for $K\in\Kn$. We assume in the following that $\Phi^*(K,\mathfrak{U})>-\infty$.

Schneider's main result in \cite{Schneider-1967} is the following general extremal property, where for $K\in\Kn$ we write
\[
K_{\rot}\coloneq\frac{w(K)}{2}B_n
\]
for the ball with center $o$ (the origin in $\Rn$) that has the same mean width as $K$. In addition, continuity and semicontinuity of functionals defined on $\Kn$ are understood with respect to the Hausdorff metric, and for convexity and concavity of functionals on $\Kn$, we consider the Minkowski addition of convex bodies.

\begin{thmx}[Schneider \cite{Schneider-1967}]\label{Schneider-main-result-A}
If $\Phi\colon \Kn\to\R$ is upper semicontinuous,
concave and increasing under set inclusion, then 
\[
\Phi^*(K,\mathfrak{U}) \leq \Phi^*(K_{\rot},\mathfrak{U})
\]
for every $K\in\Kn$ and $\mathfrak{U}\subseteq\mathscr{U}_n$.
\end{thmx}
Choosing $\Phi(K)=\vol_n(K)^{\frac{1}{n}}$ and $\mathfrak{U}=\{\Sp\}$, one immediately retrieves Urysohn's inequality, aside from its equality cases. However, Theorem~\ref{Schneider-main-result-A} has much more wide-ranging consequences. For example, Schneider used his result to elegantly demonstrate that every convex body $K$ is contained in a cylinder whose volume is at most $\frac{\vol_{n-1}(B_{n-1})}{2^{n-1}}w(K)^n$. Furthermore, a previous result of Firey and Groemer \cite{Firey-Groemer-1964} on weighted mean distances is retrieved, which, in particular, implies that any convex body of mean width one can be covered (up to translations) by any simplex whose inscribed sphere has diameter one.

\medskip

In addition to the aforementioned results, Schneider also proved the following inequality.

\begin{thmx}[Schneider \cite{Schneider-1967}]\label{Schneider-main-result-B}
If $\Psi\colon\Kn\to\R$ is lower semicontinuous,
convex and invariant under rotations, then
\[
\Psi(K) \geq \Psi(K_{\rot})
\]
for every $K\in\Kn$.
\end{thmx}

Schneider's combined proof of Theorem~\ref{Schneider-main-result-A} and Theorem~\ref{Schneider-main-result-B} is based on a symmetrization procedure involving Hadwiger rotation means of a convex body $K$, which are used to approximate $K_{\rm rot}$ (see \cite[\S 4.5.3]{hadwiger1957} or \cite[Theorem 3.3.5]{SchneiderBook}). His argument was
inspired by an analogous method of Macbeath \cite{Macbeath}, where Steiner symmetrizations were used to derive an extremal property of the ball involving inscribed polytopes with a given number of vertices. To the best of our knowledge, the equality cases of Theorems \ref{Schneider-main-result-A} and \ref{Schneider-main-result-B} have not been characterized yet.

\subsection{Outer linearizations of convex functions}
In recent years, there has been growing interest in extending geometric principles to a functional setting, motivated by developments in functional inequalities, information theory, and high-dimensional probability (see, for example, \cite{AGA-book-II, Colesanti-inbook} and the references therein). In this framework, a convex body is replaced by a convex function $\psi$ on $\Rn$ (or, equivalently, by a log-concave function $f=e^{-\psi}$ on $\Rn$), and halfspace representations are replaced by representations of $\psi$ as a supremum of affine functions. The natural language for this transition is provided by epi-convergence, infimal convolution, and Legendre--Fenchel duality.

\medskip

The purpose of this paper is twofold: on the one hand, we develop a functional generalization of Schneider’s results. At the same time, we introduce a dual paradigm for outer approximations. Instead of prescribing outer unit normals of supporting halfspaces, our approximations are controlled by slope sets in the domain of the Legendre--Fenchel transform. In this sense, the passage from convex bodies to convex functions represents a conceptual shift from geometric sampling on the sphere to dual sampling in the domain of the conjugate function. By doing so, we introduce \textit{generalized outer linearizations} of a given convex function $\psi$, which are convex minorants $q_{\psi,Y}\leq \psi$ that can be written as
\[
q_{\psi,Y}(x)\coloneq\sup\{\ell(x) :\, \ell \text{ is affine with slope } y, \ell\leq \psi, y\in Y\},\quad x\in\Rn,
\]
where $Y\subseteq \Rn$ is a (possibly infinite) set of admissible slopes. For finite slope sets, this yields classical outer linearizations from convex optimization, which are used in cutting plane methods (see, for example, \cite{Bertsekas2015,Bertsekas2011}). A crucial difference to the geometric setting is that the set of admissible slopes depends on $\psi$ itself. In particular, it turns out that the exact dependence cannot be expressed in terms of the subdifferential of $\psi$, but rather by the domain of the convex conjugate of $\psi$, resulting in a predominantly analytic and less geometric approach. This dependence makes the associated extremal problems inherently variational with respect to the dual variables.

\medskip

While it is common in functional approximation problems to impose super-coercivity on the convex functions, thereby ensuring that the convex conjugate is finite-valued everywhere, such an assumption excludes many natural examples and obscures the underlying structure. Instead, we work in the broader class of coercive convex functions, which introduces substantial analytical and topological difficulties that do not arise under super-coercivity: the domain of the Legendre--Fenchel transform may be a proper convex set, epi-convergent sequences may lose interior domain points, and pointwise convergence of conjugates generally fails on the boundary of the domain.

To address these issues, we develop a collection of tools specifically adapted to the family of coercive convex functions, which furthermore allow us to tackle stability questions associated with generalized outer linearizations. We also develop a rotational epi-symmetrization procedure for convex functions (previously introduced in a simplified setting in \cite{CLM-Hadwiger1}) and analyze its interplay with general outer linearizations.

\subsection{Overview of the paper}
Next, in Section \ref{sec:main-results}, we formulate our main results, stated as Theorem \ref{mainThmB} and Theorem \ref{mainThmA}. In Section \ref{sec:prelim}, we introduce the relevant background and notation that will be used throughout the paper, and in Section \ref{sec:epi-means-section}, we study rotation epi-means and rotational epi-symmetrizations of convex functions. In Section \ref{sec:existence}, we determine the explicit conditions needed for the existence of generalized outer linearizations of coercive convex functions. Then, in Section \ref{seq:topo}, we study topological properties and stability of generalized outer linearizations with respect to epi-convergence, for which we introduce a family of function-dependent ``balls" $b_r(\psi)$. Armed with the tools from these sections, in Section \ref{sec:proofs-main-thms} we give the full proofs of Theorem~\ref{mainThmB} and Theorem~\ref{mainThmA} and show how Theorem~\ref{Schneider-main-result-A} can be retrieved from Theorem~\ref{mainThmA}. In Section \ref{sec:log-concave}, we develop an analogous framework for generalized outer linearizations of log-concave functions and their corresponding extremal results. Finally, in Section~\ref{sec:applications}, we present several applications of our main results, including: a new functional Urysohn-type inequality; an extension of a classical covering result of Firey and Groemer \cite{Firey-Groemer-1964} to the functional setting; and an extremal result on the approximation of convex functions by piecewise affine minorants.


\section{Main results}\label{sec:main-results}

Let $\Convc$ denote the set of convex functions $\psi\colon \Rn\to (-\infty,\infty]$ that are:
\begin{itemize}
    \item \textit{coercive}: $\lim_{|x|\to\infty} \psi(x)=\infty$,
    \item \textit{proper}: there exists $x\in\Rn$ such that $\psi(x)<\infty$,
    \item and lower semicontinuous (l.s.c.).
\end{itemize}
We write $\psi_1\infconv \psi_2\in\Convc$ for the \textit{infimal convolution} or \textit{epi-sum} of $\psi_1,\psi_2\in\Convc$, that is
\[
(\psi_1\infconv \psi_2)(x)\coloneq\inf\nolimits_{y\in\Rn} \big(\psi_1(x-y)+\psi_2(y)\big),\quad x\in\Rn,
\]
and $(\lambda \sq \psi)(x)\coloneq\lambda\,\psi(x/\lambda)$, $x\in\Rn$, is the \textit{epi-multiplication} of $\psi\in\Convc$ with $\lambda >0$.

We say that $\Psi\colon \Convc\to [-\infty,\infty]$ is \emph{convex} if
\begin{equation}\label{convexity-def}
\Psi(\lambda\sq\psi_1 \infconv (1-\lambda)\sq\psi_2)\leq\lambda \Psi(\psi_1) + (1-\lambda) \Psi(\psi_2)
\end{equation}
for all $\psi_1,\psi_2\in\Convc$ and every $\lambda\in(0,1)$. For this definition to make sense, we do not allow $\Psi$ to attain both $-\infty$ and $\infty$ at the same time. Therefore, when we write that $\Psi$ takes values in $[-\infty,\infty]$, we mean that $\Psi$ takes values either in $(-\infty,\infty]$ or in $[-\infty,\infty)$.

If the inequality in \eqref{convexity-def} is reversed, then we say that $\Psi$ is \textit{concave}. Furthermore, $\Psi$ is \textit{invariant under rotations} if  $\Psi(\psi\circ\vartheta^{-1}) = \Psi(\psi)$ for every $\psi\in\Convc$ and every $\vartheta\in \SOn$. In the one-dimensional case, we generally require invariance under $\Oo=\{\pm 1\}$, and we then call a corresponding functional \textit{invariant under reflections}. Lastly, when we say that $\Psi$ is upper (or lower) semicontinuous, then this is understood with $\Convc$ being equipped with the topology associated with epi-convergence (see Section~\ref{sec:prelim} for the precise definitions).

\medskip

Our first result is the following functional version of Theorem~\ref{Schneider-main-result-B}.
\begin{theorem}
\label{mainThmB}
Let $n\geq 2$. If $\Psi\colon\Convc\to[-\infty,\infty]$ is a lower semicontinuous, convex function on $\Convc$ that is invariant under rotations, then
\[
\Psi(\psi) \geq \Psi(\psi_\rot) 
\]
for every $\psi\in\Convc$. For $n=1$, the same inequality is true if we assume $\Psi$ to be convex on $\Convc$ and invariant under reflections.
\end{theorem}

\noindent
Here, $\psi_\rot\in\Convc$ is the \emph{rotational epi-symmetrization} of $\psi$, which we define in Lemma~\ref{le:conj_rot_epi-symm} and which, for $n\geq 2$, can be obtained as a limit of a sequence of \emph{rotation epi-means} of $\psi$ (see Lemma~\ref{epi-symm-convergence}). The latter are given by
\begin{equation}
\label{eq:def_rot_epi-mean}
T_{\mathbf{\Theta}_m}(\psi)\coloneq\frac{1}{m} \sq \bigsquare_{i=1}^m(\psi\circ\vartheta_i^{-1})\in\Convc,
\end{equation}
where $\mathbf{\Theta}_m\coloneq\{\vartheta_1,\ldots,\vartheta_m\}\subseteq\SOn$, $m\in\N$, is a set of rotations. For $n=1$, we define
\begin{equation}
\label{eq:def_rot_onedim}
\psi_\rot \coloneq \frac 12 \sq \big(\psi \infconv \psi^-\big)\in\Convc,
\end{equation}
where $\psi^-(x)\coloneq\psi(-x)$ for $x\in\Rn$. These symmetrizations, which generalize the classical Hadwiger rotation means,
were previously considered by Colesanti, Ludwig, and Mussnig in \cite[Section 4.1]{CLM-Hadwiger1}. There, the smaller space of super-coercive convex functions on $\Rn$ was used, and we will show in Section~\ref{sec:epi-means-section} that this definition extends to $\Convc$.

\medskip

Next, let us introduce our functional analogue of Theorem~\ref{Schneider-main-result-A}. Given $\psi\in\Convc$ and $y\in\Rn$, we write $\ell_{\psi,y}\colon\Rn\to [-\infty,\infty)$ for the largest affine function with gradient $y$ that is bounded from above by $\psi$. This means that
\[
\ell_{\psi,y}(x)=\langle x,y\rangle +c,\quad x\in\Rn,
\]
where $c=c_\psi(y)\in [-\infty,\infty)$ is the maximal number such that $\ell_{\psi,y}\leq \psi$ pointwise. In the case no real $c$ satisfies this condition, we set $c=-\infty$ (and thus $\ell_{\psi,y}\equiv -\infty$), and we remark that $c=\infty$ cannot occur since $\psi$ is proper. For nonempty $Y\subseteq \Rn$, we now define $q_{\psi,Y}\colon \Rn\to [-\infty,\infty]$ as the pointwise supremum
\begin{equation}
\label{eq:def_outer_lin}
q_{\psi,Y}(x)\coloneq\sup\nolimits_{y\in Y} \ell_{\psi,y}(x),\quad x\in\Rn.
\end{equation}
Note that $q_{\psi,Y}$ is convex and $q_{\psi,Y}\leq \psi$ pointwise. Furthermore, if $Y$ is finite, then $q_{\psi,Y}$ is piecewise affine. Observe that if $y$ is an element of the \textit{subdifferential}
\[
\partial \psi(x)=\{y\in\Rn:\, \psi(z)\geq \psi(x)+\langle y,z-x\rangle \text{ for all } z\in\Rn\}
\]
for some $x\in\Rn$, then trivially $\ell_{\psi,y}>-\infty$. When $Y$ is a finite set consisting of such vectors, the function $q_{\psi,Y}$ is also known as the \textit{outer linearization} of $\psi$, which is closely related to cutting plane methods from convex optimization; see, for example, \cite[Section 4.1]{Bertsekas2015}. Given that we allow a more general construction here (cf.\ also Remark~\ref{rem:outer_lin_subdiff} below), we call \eqref{eq:def_outer_lin} the \textit{generalized outer linearization} of $\psi$. Note that in particular, since $\psi\in\Convc$, we have \begin{equation}\label{Y-equals-Rn}
    q_{\psi,\R^n}=\psi
\end{equation}which follows from Lemma~\ref{le:q_psi_rn_psi} below. Related results on the reconstruction of a convex function from slope data were, for example, obtained in \cite{benoist_daniilidis_2002,benoist_daniilidis_2005} using cyclically monotone operators (cf.\ \cite[Theorem 24.9]{RockafellarBook}).

We will show in Proposition~\ref{prop:C_intersection} that for any given $\psi\in\Convc$,  the function $q_{\psi,Y}$ is also an element of $\Convc$ if and only if $Y\subseteq\Rn$ contains a subset $Y_o$ such that
\begin{equation}
\label{eq:cond_existence}
\pos(Y_o)=\Rn \qquad \text{and} \qquad Y_o\subseteq \dom(\Leg \psi).
\end{equation}
Here, $\dom(\Leg \psi)$ denotes the \textit{domain} of the Legendre--Fenchel transform $\Leg \psi$ of $\psi$ (see Section~\ref{sec:prelim}). Let us note that $\dom(\Leg \psi)=\Rn$ when $\psi$ is \textit{super-coercive}, that is $\lim_{|x|\to\infty} \psi(x)/|x|=\infty$, or equivalently, when the range of the subdifferential of $\psi$ is $\Rn$. In this case, the conditions \eqref{eq:cond_existence} are therefore equivalent to $\pos(Y)=\Rn$. We furthermore note that exact conditions on $Y$ so that $q_{\psi,Y}\in\Convc$ cannot be formulated in terms of the subdifferential $\partial \psi$ and that it is necessary to consider the convex conjugate $\Leg \psi$ (see Remark~\ref{rem:outer_lin_subdiff}).

To ensure that all expressions are well-defined in the following (see Lemma~\ref{le:existence_consequences} and Remark~\ref{re:conditions_necessary}), we have to further restrict the admissible sets $Y\subseteq\Rn$. Given $\psi\in\Convc$, we define
\[
\Ycoll_\psi\coloneq\{Y\subseteq\Rn :\, \exists Y_o\subseteq Y \text{ s.t. } \pos(Y_o)=\Rn, Y_o\subseteq \interior(r(\psi) B_n)\},
\]
where $\interior(A)$ denotes the interior of $A\subseteq\Rn$, and $r(\psi)\in(0,\infty]$ is the radius of the largest open, centered ball that is contained in $\dom(\Leg \psi)$ (see Section~\ref{sec:prelim}). Now, for $\Phi\colon\Convc\to [-\infty,\infty]$, a convex function $\psi\in\Convc$, and an admissible collection  
$\mathcal{Y}\subseteq \Ycoll_\psi$,  we define
\begin{equation}
\label{mainDef} 
\Phi^*(\psi, \mathcal{Y}) \coloneq \inf\{ \Phi(q_{\psi\circ \vartheta^{-1},Y}) : \,\vartheta\in \SOn,\, Y\in \mathcal{Y}\},
\end{equation}
in case $n\geq 2$. For $n=1$, we define $\Phi^*$ analogously, considering $\Oo$ in the infimum instead.

\medskip

Our following main result corresponds to Theorem~\ref{Schneider-main-result-A}. For this, we say that a functional $\Phi$ on $\Convc$ is \textit{monotone decreasing} if $\Phi(\psi_1)\leq \Phi(\psi_2)$ whenever $\psi_1,\psi_2\in\Convc$ are such that $\psi_1\geq\psi_2$ pointwise. Under such a monotonicity assumption on $\Phi$ and given that $q_{\psi,Y} \leq \psi$ always holds, the above definition of $\Phi^*(\psi)$ can also be viewed as a kind of best approximation of $\Phi(\psi)$ via outer linearizations. This becomes particularly evident if $\Phi$ is also invariant under rotations.

\begin{theorem}
\label{mainThmA}
If $\Phi\colon \Convc\to[-\infty,\infty]$ is upper semicontinuous, concave, and monotone decreasing, then
\[
\Phi^*(\psi,\mathcal{Y}) \leq \Phi^*(\psi_{\rot},\mathcal{Y})
\]
for every $\psi\in\Convc$ and $\mathcal{Y}\subseteq \Ycoll_\psi$. In particular, $q_{\psi\circ \vartheta^{-1},Y},q_{\psi_{\rot}\circ \vartheta^{-1},Y}\in\Convc$ for every $Y\in\mathcal{Y}$ and $\vartheta\in\On$, which ensures that the above expressions are well-defined.
\end{theorem}
We note that in several applications, the above restriction to function-dependent slope sets in $\Ycoll_\psi$ can be bypassed relatively easily. An example is presented in Section~\ref{suse:approx}.

\medskip

The main results above show that rotational epi-symmetrization plays the same extremal role in the functional setting considered here as the Euclidean ball does in Schneider’s principle. Broadly speaking, the proof strategies for Theorem~\ref{mainThmB} and Theorem~\ref{mainThmA} roughly follow Schneider's ideas in \cite{Schneider-1967}. Our proof method uses a functional symmetrization process involving rotation epi-means \cite{CLM-Hadwiger1}, which are functional analogues of Hadwiger's classical rotation means of convex bodies used by Schneider in his original proofs of Theorems \ref{Schneider-main-result-A} and \ref{Schneider-main-result-B}. However, considering the commonly used space of coercive convex functions leads to considerable complications, mostly concerning topological properties of generalized outer linearizations. We emphasize that these issues would not occur in the more restrictive setting of super-coercive functions. This is already evident from the fact that for super-coercive, convex $\psi$, the conditions on $\mathcal{Y}$ so that $\Phi^*(\psi,\mathcal{Y})$ exists no longer depend on $\psi$. An essential step in proving Theorem~\ref{mainThmA} is therefore found in Proposition~\ref{upper-semi-lem}, where we show that $\Phi^*(\cdot,\mathcal{Y})$ is upper semicontinuous. Note that Proposition~\ref{upper-semi-lem} is more general than is necessary for the proof of Theorem~\ref{mainThmA}, and that we consider the upper semicontinuity of $\Phi^*(\cdot,\mathcal{Y})$ to be of independent interest.
 
By embedding convex bodies as indicator functions within our functional framework and restricting to a corresponding subclass of admissible linearizations, we recover Schneider’s Theorems \ref{Schneider-main-result-A} and \ref{Schneider-main-result-B} (see Subsection \ref{geometric-special-case} for the details). For results on geometric or functional extremal problems that are most closely related to those of this paper, we refer the reader to \cite{BucurFragalaLamboley,Ting-Chen,Hoehner-2023,Paouris-Pivovarov,PB-2020,Rinott,Schneider-1967} and the references therein.

\medskip

For the reader's convenience, the main correspondences between convex bodies and convex functions that arise in this paper are summarized in Table~\ref{tab:corr}.

\begin{table}[!ht]
\centering
\footnotesize
\setlength{\tabcolsep}{5pt}
\renewcommand{\arraystretch}{1.2}

\begin{tabular}{c|c}
\textbf{Convex bodies} & \textbf{Convex functions} \\
\hline

Convex body $K \in\Kn$ 
& Convex function $\psi\in\Convc$ \\

Minkowski sum $K+L$ & Infimal convolution $\psi\square\varphi$\\

Volume $\vol_n(K)=\int_{\Rn}e^{-\ind_K(x)}\dint x$ & Total mass $J(\psi)=\int_{\Rn}e^{-\psi(x)}\dint x$\\

Outer normal directions $U\subseteq\Sp$
& Slope set $Y\subseteq \Rn$ \\

Support function $h_K(u)$ 
& Legendre--Fenchel transform $\Leg\psi(y)$ \\

Intersection of halfspaces 
$\displaystyle K=\bigcap_{u\in U} H^{-}(K,u)$
& Supremum of affine functions  
$\displaystyle \psi=\sup_{y\in Y} (\langle x,y\rangle-\Leg\psi(y))$ \\

Circumscribed ``polyhedral" set  $P(K,U)$
& Generalized outer linearization $q_{\psi,Y}$ \\

$\pos(U)=\Rn$
& $\exists Y_o\subseteq Y$ such that $\pos(Y_o)=\Rn$ and $Y_o\subseteq\dom(\Leg\psi)$ \\

Hausdorff convergence
& Epi-convergence \\

Euclidean balls $r B_n$
& Function-dependent balls $b_r(\psi)$ \\

Outer parallel body $K+rB_n$
& Outer parallel function $\psi \square b_r(\psi)$ \\

Hadwiger rotation means
& Rotation epi-means $T_{\mathbf{\Theta}_m}(\psi)$ \\

Ball $K_{\rot}$
& Rotational epi-symmetrization $\psi_{\rot}$ \\

Monotone concave functionals on $\Kn$
& Monotone concave functionals on $\Convc$ \\

Theorem~\ref{Schneider-main-result-A}
& Theorem~\ref{mainThmB} \\

Theorem~\ref{Schneider-main-result-B} & Theorem~\ref{mainThmA}\\
Urysohn's inequality & Theorem~\ref{urysohn-2}
\end{tabular}

\caption{\label{tab:corr}Dictionary between outer approximation theory for convex bodies and generalized outer linearizations of convex functions.}
\end{table}


\section{Background and preliminaries}
\label{sec:prelim}
For a convex function $\psi\colon \Rn\to(-\infty,\infty]$, let
\[
\Leg \psi(x)\coloneq\sup\nolimits_{y\in\Rn} \big(\langle x,y \rangle - \psi(y) \big),\quad x\in\Rn,
\]
denote the \textit{convex conjugate} or \textit{Legendre--Fenchel transform} of $\psi$. The result below can be found, for example, in \cite[Theorem 1.6.13]{RockafellarBook}.

\begin{lemma}
\label{le:conj_is_nice}
If $\psi\colon \Rn\to (-\infty,\infty]$ is a proper, lower semicontinuous, convex function, then $\Leg \psi$ defines a proper, lower semicontinuous, convex function on $\Rn$ and $\Leg\Leg \psi = \psi$.
\end{lemma}

Lemma~\ref{le:conj_is_nice} shows that $\Leg$ is an involution, and we note that this map is also order-reversing (considering the partial order given by pointwise inequalities between functions). In fact, the Legendre--Fenchel transform is, essentially, the only order-reversing involution on the space of convex functions considered above \cite{Artstein-Milman2009}.

We also need the following criterion for a convex function to be coercive, which can be found in \cite[Theorem 11.8]{Rockafellar-Wets}. Here and throughout the paper,
\[
\dom(\psi)\coloneq\{x\in\Rn :\, \psi(x)<\infty\}
\]
denotes the \textit{domain} of a convex function $\psi$ and we write $o$ for the origin in $\Rn$.
\begin{lemma}
\label{le:conj_coercive}
A proper, lower semicontinuous, convex function $\psi\colon \Rn\to (-\infty,\infty]$ is coercive if and only if $o\in \interior(\dom(\Leg \psi))$.
\end{lemma}

For $\psi\in\Convc$, we define
\begin{equation}
\label{eq:def_r_psi}
r(\psi)\coloneq\sup\{r\geq 0:\, rB_n \subseteq \dom(\Leg \psi)\}\in (0,\infty].
\end{equation}
If $r(\psi)=\infty$, then clearly $\dom(\Leg \psi)=\Rn$. If  $r(\psi)$ is finite, then it is the radius of the largest, open, centered ball that is contained in $\dom(\Leg \psi)$. In particular, Lemma~\ref{le:conj_coercive} shows that $r(\psi)$ is well-defined and strictly positive. We will repeatedly need to consider the Euclidean unit ball with radius $r(\psi)$, and we use the convention $r(\psi)B_n=\Rn$ in the case $r(\psi)=\infty$.

\medskip

Next, we consider the interplay between convex conjugation and infimal convolution. The following is a consequence of \cite[Theorem 1.6.17]{SchneiderBook} together with the definition of $\Convc$.
\begin{lemma}
\label{le:infconv_conj}
If $\psi_1,\psi_2\in \Convc$, then $\Leg(\psi_1 \infconv \psi_2) = \Leg \psi_1 + \Leg \psi_2$.
\end{lemma}

We furthermore need to consider the action of the Legendre--Fenchel transform on support functions of convex bodies. For every $K\in\Kn$, we have
\[
\Leg h_K (x) = \ind_K(x) \coloneq\begin{cases}
    0\quad &\text{if } x\in K,\\
    \infty \quad &\text{else},
\end{cases}
\]
which is the \emph{convex indicator function} of $K$. The special case $K=\{y\}$ with $y\in\Rn$ gives
\begin{equation}
\label{eq:conj_linear_fct}
\Leg(x\mapsto \langle x,y\rangle) = \Leg h_{\{y\}}= \ind_{\{y\}}.
\end{equation}

We consider convex functions on $\Rn$ together with the topology associated with \textit{epi-convergence} (see, for example, \cite[Chapter 7]{Rockafellar-Wets}). Here, we say that a sequence of functions $\psi_j\colon\Rn\to (-\infty,\infty]$, $j\in\N$, \emph{epi-converges} to $\psi\colon\Rn\to(-\infty,\infty]$ if for every $x\in\Rn$:
\begin{itemize}
    \item $\psi(x)\leq \liminf_{j\to\infty}\psi_j(x_j)$ for every sequence $x_j\in\Rn$, $j\in\N$, such that $x_j\to x$, and
    \item $\psi(x)=\lim_{j\to\infty}\psi_j(x_j)$ for some sequence $x_j\in\Rn$, $j\in\N$, such that $x_j\to x$.
\end{itemize}
In this case, we write $\psi=\epilim_{j\in\N} \psi_j$. On $\Convc$, epi-convergence is roughly equivalent to Hausdorff convergence of level sets (see \cite[Lemma 5]{CLM-IMRN} and \cite[Theorem 3.1]{Beer-Rockafellar-Wets-1992}). More generally, we need the following result, which can be found in \cite[Theorem 7.17]{Rockafellar-Wets}.
\begin{lemma}
\label{le:epi_conv_pointwise}
Let $\psi_j\colon \Rn\to(-\infty,\infty]$, $j\in\N$, be a sequence of convex functions. If $\psi\colon \Rn\to(-\infty,\infty]$ is a lower semicontinuous, convex function such that $\dom(\psi)$ has nonempty interior, then $\psi=\epilim_{j\in\N} \psi_j$ if and only if $\psi_j$ converges pointwise to $\psi$ on a dense subset of $\Rn$. Equivalently, $\psi_j$ converges uniformly to $\psi$ on every compact set that does not contain a boundary point of $\dom(\psi)$.
\end{lemma}

Lastly, we state \cite[Theorem 11.34]{Rockafellar-Wets}, which shows that convex conjugation is continuous with respect to epi-convergence.
\begin{lemma}
\label{le:conj_continuous}
If $\psi,\psi_j\colon \Rn\to(-\infty,\infty]$, $j\in\N$, are lower semicontinuous, proper, convex functions, then $\psi=\epilim_{j\in\N} \psi_j$ if and only if $\Leg \psi = \epilim_{j\in\N} \Leg \psi_j$.
\end{lemma}


\section{Rotation epi-means and rotational epi-symmetrizations}\label{sec:epi-means-section}
Rotation epi-means and rotational epi-symmetrizations were introduced in \cite[Section 4.1]{CLM-Hadwiger1} for super-coercive convex functions on $\Rn$. We will treat these constructions thoroughly in the following, since additional technicalities arise on the larger space of coercive convex functions, and since in \cite{CLM-Hadwiger1} no detailed proofs were given.

In the next result, we use integration with respect to the Haar probability measure on $\SOn$. 

\begin{lemma}
    \label{le:conj_rot_epi-symm}
    Let $n\geq 2$. For every $\psi\in\Convc$, there exists a unique function $\psi_\rot\in\Convc$, the \emph{rotational epi-symmetrization} of $\psi$, such that
    \begin{equation}
    \label{eq:rot_mean_conj}
    \Leg\psi_\rot(x)= \int_{\SOn} \Leg\psi(\vartheta^{-1} x) \dint \vartheta
    \end{equation}
    for $x\in\Rn$.
\end{lemma}
\begin{proof}
Throughout the proof, fix $\psi\in \Convc$ and let $\varphi\colon \Rn\to(-\infty,\infty]$ denote the map defined by the integral on the right-hand side of \eqref{eq:rot_mean_conj}. By Lemma~\ref{le:conj_is_nice} and Lemma~\ref{le:conj_coercive}, we need to show that $\varphi$ is well-defined and that it is a proper, lower semicontinuous, convex function with $o\in \interior(\dom(\varphi))$.

Since $\Leg \psi$ is a proper convex function, it does not attain the value $-\infty$. Thus, for every $x\in\Rn$, the map $\vartheta\mapsto \Leg\psi(\vartheta^{-1} x)$ is measurable and takes values in $(-\infty,\infty]$. Since $\SOn$ is compact, the negative part of this function has finite integral and $\eqref{eq:rot_mean_conj}$ is well-defined with values in $(-\infty,\infty]$.

Next, let $\lambda\in(0,1)$ and $x,y\in\Rn$. Assume first that $\varphi(x),\varphi(y)<\infty$, which is only possible if $\vartheta^{-1}x,\vartheta^{-1}y\in \dom(\Leg \psi)$ for a.e.\ $\vartheta\in\SOn$. Since $\dom(\Leg \psi)$ is convex, it follows that also $\vartheta^{-1}(\lambda x + (1-\lambda)y))\in \dom(\Leg \psi)$ for a.e.\ $\vartheta\in\SOn$. Thus, by the convexity of $\Leg \psi$,
\begin{align*}
\varphi(\lambda x + (1-\lambda)y) &= \int_{\SOn} \Leg \psi (\vartheta^{-1}[\lambda x + (1-\lambda)y]) \dint\vartheta\\
&= \int_{\SOn} \Leg \psi(\lambda \vartheta^{-1}x + (1-\lambda) \vartheta^{-1}y) \dint \vartheta\\
&\leq \int_{\SOn} \Big( \lambda \Leg \psi(\vartheta^{-1} x) + (1-\lambda) \Leg \psi(\vartheta^{-1} y) \Big) \dint \vartheta\\
&= \lambda \varphi(x) + (1-\lambda) \varphi(y).
\end{align*}
In the remaining case $\varphi(x)=\infty$ or $\varphi(y)=\infty$, and we trivially have 
\[
\varphi(\lambda x + (1-\lambda)y) \leq \lambda \varphi(x)+(1-\lambda)\varphi(y)=\infty.
\]
Hence, $\varphi$ is convex.

Since $\psi$ is coercive, it follows from Lemma~\ref{le:conj_coercive} that $o\in\interior(\dom(\Leg \psi))$ and $r(\psi)>0$ (cf.\ \eqref{eq:def_r_psi}). Thus, it  follows from \eqref{eq:rot_mean_conj} that $\varphi$ is finite on $\interior(r(\psi) B_n)$. In particular, we have $o\in \interior(\dom(\varphi))$ and $\varphi$ is proper.

It remains to show that $\varphi$ is lower semicontinuous. Fix $\bar{x}\in \Rn$ and let $x_j$, $j\in\N$, be a sequence in $\Rn$ with limit $\bar{x}$. We can find $R>0$ such that $x_j\in R B_n$ for every $j\in\N$, and by the properties of $\Leg \psi$ there exists $m=\min_{x\in R B_n} \Leg \psi(x)\in\R$. It is easy to see that $\Leg \psi-m$ is nonnegative on $R B_n$, and by \eqref{eq:rot_mean_conj} also $\varphi-m$ is nonnegative on $R B_n$. For $j\in\N$, we now define $f_j\colon \SOn\to \R$ as
\[
f_j(\vartheta)=\Leg \psi(\vartheta^{-1} x_j) - m.
\]
Observe that since $\vartheta^{-1}x_j\in R B_n$ for every $\vartheta\in \SOn$ and $j\in\N$, the functions $f_j$ are nonnegative due to our choice of $m$. Thus, by Fatou's lemma and the lower semicontinuity of $\Leg \psi$, we obtain
\begin{align*}
\liminf_{j\in\N} \varphi(x_j) &= \liminf_{j\in\N} \int_{\SOn} f_j(\vartheta) \dint \vartheta + m\\
&\geq \int_{\SOn} \liminf_{j\in\N} f_j(\vartheta) \dint\vartheta + m\\
&= \int_{\SOn} \liminf_{j\in\N} \left(\Leg \psi(\vartheta^{-1} x_j) - m \right) \dint \vartheta + m\\
&\geq \int_{\SOn} \left(\Leg \psi(\vartheta^{-1} \bar{x}) - m \right)\dint\vartheta + m\\
&= \varphi(\bar{x}),
\end{align*}
which shows that $\varphi$ is lower semicontinuous.
\end{proof}

\begin{remark}
    The statement of Lemma~\ref{le:conj_rot_epi-symm} remains valid in the one-dimensional case if we replace Haar averaging on ${\rm SO}(1)$ with averaging over $\Oo=\{\pm 1\}$. Indeed, for $n=1$, it follows from \eqref{eq:def_rot_onedim} together with Lemma~\ref{le:infconv_conj} that for $x\in\R$,
    \[
    \Leg \psi_{\rot}(x)=\frac{\Leg \psi(x)+\Leg \psi(-x)}{2}.
    \]
\dssymb
\end{remark}

Recall that for $\psi\in\Convc$ and $\mathbf{\Theta}_m=\{\vartheta_1,\ldots,\vartheta_m\}\subseteq\SOn$ with $m\in\N$, the \emph{rotation epi-mean} $T_{\mathbf{\Theta}_m}(\psi)\in\Convc$ is defined as
\begin{equation*}
T_{\mathbf{\Theta}_m}(\psi)\coloneq\frac{1}{m} \sq \bigsquare_{i=1}^m(\psi\circ\vartheta_i^{-1}).
\end{equation*}

\begin{lemma}\label{epi-symm-convergence}
    Let $n\geq 2$. For every $\psi\in\Convc$, there exists a sequence of rotation epi-means of $\psi$ which epi-converges to $\psi_{\rot}\in\Convc$.
\end{lemma}
\begin{proof}
Fix $\psi\in \Convc$ and assume first that $r(\psi)<\infty$. By Lemma~\ref{le:conj_coercive} and \eqref{eq:def_r_psi}, there exists $u_o\in\Sp$ such that $ru_o \in \dom(\Leg \psi)$ for every $r\in[0,r(\psi))$ and $ru_o \notin \dom(\Leg \psi)$ for every $r\in (r(\psi),\infty)$. By convexity, this implies that
\[
\dom(\Leg \psi) \subseteq \{x\in\Rn :\,\langle x,u_o\rangle \leq r(\psi)\}.
\]
Thus, for every $r\in(r(\psi),\infty)$ there exists an open neighborhood $U(r)\subseteq \Sp$ of $u_o$ such that
\begin{equation}
\label{eq:ru_notin_dom}
r u \notin \dom(\Leg \psi)
\end{equation}
for every $u\in U(r)$ and every $r\in (r(\psi),\infty)$. It now follows from \eqref{eq:rot_mean_conj} that the inclusion
$$\interior (r(\psi) B_n) \subseteq \dom (\Leg \psi_{\rot})\subseteq r(\psi) B_n$$
holds true.

Next, observe that for every $m\in\N$ and every set of rotations $\mathbf{\Theta}_m=\{\vartheta_1,\ldots,\vartheta_m\}\subseteq\SOn$, $m\in\N$, we have
\begin{equation}
\label{eq:rot_mean_seq_conj}
\Leg(T_{\mathbf{\Theta}_m}(\psi))(x)=\frac{1}{m} \sum_{i=1}^m \Leg \psi(\vartheta_i^{-1} x)
\end{equation}
for $x\in\Rn$. We now want $\mathbf{\Theta}_m\subseteq\SOn$, $m\in\N$, to be a sequence such that \eqref{eq:rot_mean_seq_conj} approximates the integral \eqref{eq:rot_mean_conj} as $m\to\infty$. The existence of such a sequence can be seen from the representation of the Haar measure on $\SOn$ using hyperspherical coordinates (see, for example, \cite[Section 1.2]{Meckes-2019}). Since $\SOn$ is compact and acts transitively on $\Sp$, we may in fact choose such a sequence so that for every $\varepsilon >0$, there exists $m_o(\varepsilon)\in\N$ such that for every $u\in\Sp$ and $m\geq m_o(\varepsilon)$, there exists $\vartheta\in \mathbf{\Theta}_m$ with
\begin{equation}
\label{eq:close_to_u_o}
|u_o-\vartheta^{-1}u|<\varepsilon.
\end{equation}
Since convex functions are continuous on the interior of their domains (see, e.g., \cite[Theorem 25.5]{RockafellarBook}), it follows from \eqref{eq:def_r_psi} that $\Leg(T_{\mathbf{\Theta}_m}(\psi))$ converges pointwise to $\Leg \psi_{\rot}$ on $\interior( r(\psi) B_n)$ as $m\to\infty$. Furthermore, if $x\in\Rn$ is such that $|x|>r(\psi)$, then by \eqref{eq:close_to_u_o} we can find $m_o\in\N$ such that for every $m\geq m_o$, there exists $\vartheta\in \mathbf{\Theta}_m$ so that $\vartheta^{-1}x/|x|$ lies in the neighborhood $U(|x|)$ of $u_o$. By \eqref{eq:ru_notin_dom} this implies $\Leg(T_{\mathbf{\Theta}_m}(\psi))(x)=\infty$ for every $m\geq m_o$. We conclude that $\Leg(T_{\mathbf{\Theta}_m}(\psi))$ converges pointwise to $\psi_{\rot}$ on
\[
\{x\in\Rn :\, |x|\neq r(\psi)\},
\]
which is a dense subset of $\Rn$. Thus, it follows from Lemma~\ref{le:epi_conv_pointwise} that $\Leg(T_{\mathbf{\Theta}_m}(\psi))$ epi-converges to $\Leg \psi_{\rot}$, which, by Lemma~\ref{le:conj_continuous} is equivalent to the epi-convergence of $T_{\mathbf{\Theta}_m}(\psi)$ to $\psi_{\rot}$ as $m\to\infty$.

In the case $r(\psi)=\infty$, we have $\dom (\psi)=\dom (\psi_{\rot})=\Rn$. We can therefore proceed using a simplified version of the argument above and only need to choose $\mathbf{\Theta}_m\subseteq \SOn$, $m\in\N$, such that \eqref{eq:rot_mean_seq_conj} approximates \eqref{eq:rot_mean_conj}.
\end{proof}


\section{Existence of generalized outer linearizations}\label{sec:existence}
In this section, we analyze conditions for the existence of generalized outer linearizations, as defined in \eqref{eq:def_outer_lin}. More precisely, in Proposition~\ref{prop:C_intersection} we establish exact conditions on $Y\subseteq \Rn$ so that $q_{\psi,Y}$ is an element of $\Convc$. The succeeding Lemma~\ref{le:existence_consequences} then deals with conditions so that all of the quantities that (implicitly) appear in the statement of Theorem~\ref{mainThmA} are well-defined. 

\medskip

For a family of lower semicontinuous, convex functions $\psi_i\colon \Rn\to (-\infty,\infty]$, $i\in I$, we denote by
\[
\hat{\inf}_{i\in I}\{\psi_i\} = \sup \{\psi\colon \Rn \to \R :\, \psi \text{ is l.s.c.\ and convex and } \psi\leq \psi_i,\; \forall i\in I\}
\]
its \emph{regularized infimum}, which is the largest convex function that is pointwise bounded from above by the functions of the family. The following result is a consequence of \cite[Lemma 2]{Artstein-Milman2009}.
\begin{lemma}
\label{le:conj_wedge}
If $\psi_i\colon \Rn\to (-\infty,\infty]$, $i\in I$, is a family of lower semicontinuous, convex functions, then
$$\Leg\left(\hat{\inf}_{i\in I} \{\psi_i\}\right) = \sup\nolimits_{i\in I}\left\{\Leg \psi_i\right\}.$$
\end{lemma}

Recall that for $\psi\in\Convc$ and $y\in\Rn$, we define $\ell_{\psi,y}\colon\Rn\to [-\infty,\infty)$ as the largest affine function with gradient $y$ that is bounded from above by $\psi$. Furthermore, for $Y\subseteq \Rn$, we consider the pointwise supremum $q_{\psi,Y}\coloneq\sup_{y\in Y} \ell_{\psi,y}$. In addition, we will denote by $\conv(A)$ the \textit{convex hull} of $A\subseteq \R^k$. The next result can essentially also be found in \cite[Section 2]{Bertsekas2011}, where it was shown in a slightly more restrictive setting; see also \cite[Section 4.3]{Bertsekas2015}. 

\begin{lemma}
\label{le:vector_dom}
If $\psi\in\Convc$ and $y\in\Rn$, then
\[
\ell_{\psi,y}(x)=\langle x,y\rangle -\Leg \psi(y)
\]
for every $x\in\Rn$. In particular, the function $\ell_{\psi,y}$ is proper if and only if
\begin{equation}
\label{eq:y_dom_leg_psi}
y\in\dom(\Leg \psi).
\end{equation}
Moreover, if $Y\subseteq \Rn$ is such that $Y\cap \dom(\Leg \psi)\neq \emptyset$, then
\[
\epi(\Leg q_{\psi,Y}) = \operatorname{cl}\big(\conv\{(y,t)\in\Rn\times\R:\, y\in Y \text{ s.t. } y\in \dom(\Leg \psi), \Leg \psi(y)\leq t\}\big).
\]
\end{lemma}
\begin{proof}
By definition,
\[
\ell_{\psi,y}(x)=\langle x,y\rangle +c
\]
for $x\in\Rn$, where $c\in [-\infty,\infty)$ is chosen maximal so that $\ell_{\psi,y}\leq \psi$ pointwise. By the properties of convex conjugation, in particular \eqref{eq:conj_linear_fct}, we have
\[
\Leg \ell_{\psi,y} = \ind_{\{y\}} - c,
\]
where $c$ is maximal such that this function is pointwise greater than or equal to $\Leg\psi$. This shows that $c\neq -\infty$ (or equivalently, $c\in\R$) if and only if \eqref{eq:y_dom_leg_psi} is satisfied, in which case $c=-\Leg\psi(y)$.

Now let $Y\subseteq \Rn$ be nonempty and assume without loss of generality that all functions $\ell_{\psi,y}$ are proper (since $q_{\psi,Y}$ is defined as a supremum), or, equivalently, $Y\subseteq \dom(\Leg \psi)$. By the first part of the proof together with Lemma~\ref{le:conj_wedge}, we have
\[
q_{\psi,Y} = \Leg \left(\hat{\inf}_{y\in Y} \{\Leg \ell_{\psi,y}\} \right) = \Leg \left(\hat{\inf}_{y\in y} \left\{\ind_{\{y\}} + \Leg\psi(y)\right\} \right).
\]
The statement now follows if we apply the Legendre--Fenchel transform to the equation above and consider the fact that for any function $\varphi\colon \Rn \to (-\infty,\infty]$, the function $\Leg \Leg \varphi$ is the largest lower semicontinuous convex function that is pointwise dominated by $\varphi$ (i.e., $\Leg \Leg \varphi$ is the closed convex hull of $\varphi$).
\end{proof}

We point out the following simple consequence of Lemma~\ref{le:vector_dom}.

\begin{lemma}
\label{le:q_psi_rn_psi}
If $\psi\in\Convc$, then $q_{\psi,\Rn}=\psi$.
\end{lemma}
\begin{proof}
Since $\psi\in\Convc$, it follows from Lemma~\ref{le:conj_is_nice} that $\Leg\Leg\psi=\psi$. Thus, by the definition of $q_{\psi,\Rn}$ in \eqref{eq:def_outer_lin} and Lemma~\ref{le:vector_dom} we have
\[
q_{\psi,\Rn}(x)=\sup\nolimits_{y\in\Rn}\left(\langle x,y\rangle-\Leg\psi(y)\right)=\Leg\Leg\psi(x)=\psi(x)
\]
for every $x\in\Rn$.
\end{proof}

The following result is a monotonicity property of the generalized outer linearization. 
\begin{lemma}\label{q-is-monotone-wrt-inclusion}
    If $Y_1,Y_2\subseteq\Rn$ are such that $Y_1\subseteq Y_2$, then $q_{\psi,Y_1}\leq q_{\psi,Y_2}$ for every $\psi\in\Convc$.
\end{lemma}

\begin{proof}
    Since $Y_1\subseteq Y_2$, for every $x\in\Rn$ we have 
    \[
\{\langle x,y\rangle -\Leg\psi(y):\,y\in Y_1\}\subseteq\{\langle x,y\rangle -\Leg\psi(y):\,y\in Y_2\}.
    \]
    Thus, by definition,
    \[
q_{\psi,Y_1}(x) = \sup\{\langle x,y\rangle -\Leg\psi(y):\,y\in Y_1\} \leq \sup\{\langle x,y\rangle -\Leg\psi(y):\,y\in Y_2\}=q_{\psi,Y_2}(x).
    \]
\end{proof}

\begin{remark}
\label{rem:outer_lin_subdiff}
It is a consequence of \cite[Proposition 11.3]{Rockafellar-Wets} that condition \eqref{eq:y_dom_leg_psi} is met for every $y\in\Rn$ such that $y\in \partial \psi(\Rn)$.
Conversely, \cite[Lemma 1.6.16]{SchneiderBook} shows that if $y$ is contained in the relative interior of $\dom(\Leg \psi)$, then $y\in\partial \psi(\Rn)$. If, however, $y$ is a boundary point of $\dom(\Leg \psi)$, then $y$ may or may not be an element of $\partial \psi(\Rn)$. An example of the latter is given by $\psi(x)\coloneq\sqrt{1+|x|^2}$, which satisfies
\[
\partial \psi(\Rn)=\{x\in\Rn :\, |x|<1\}.
\]
It is straightforward to check that $\Leg\psi(x)=-\sqrt{1-|x|^2}+\ind_{B_n}(x)$, and, in particular,\linebreak$\dom(\Leg \psi)=B_n$. Now if $y\in\Rn$ is such that $|y|=1$, then $y\notin \partial \psi(\Rn)$ but $\ell_{\psi,y}$ is proper. Thus, in order to determine whether $\ell_{\psi,y}$ is proper, it is not sufficient to verify whether $y\in\partial \psi(\Rn)$.

Let us also mention that the example above suggests that it may be sufficient to verify whether $y$ is in the closure of $\partial \psi(\Rn)$. However, this is also not the case, as the example
\[
\varphi(x)\coloneq\begin{cases}0 \qquad &\text{if } |x|\leq 1\\ |x|-\sqrt{|x|}\qquad &\text{if } 1<|x| \end{cases}
\]
shows. Again, $\partial \varphi(\Rn)=\{x\in \Rn :\,|x|<1\}$ but, in contrast to the example above, also $\dom(\Leg \varphi)=\{x\in \Rn : |x|<1\}$.\dssymb
\end{remark}

The following is obtained from \cite[Theorem 1.1.14]{SchneiderBook}; see also \cite[Section 1.3, Note 4]{SchneiderBook}.

\begin{lemma}
\label{le:pos_conv_origin}
If $B\subseteq \Rn$, then $\pos(B) = \Rn$ if and only if $\conv(B)$ contains the origin in its interior.
\end{lemma}

We can now state and prove the main result of this section.

\begin{proposition}
\label{prop:C_intersection}
Let $\psi\in\Convc$ and let $Y\subseteq \Rn$ be nonempty. The function $q_{\psi,Y}$ is an element of $\Convc$ if and only if there exists $Y_o\subseteq Y$ such that
\begin{equation}
\label{eq:cond_Y_o}
Y_o \subseteq \dom(\Leg \psi) \quad \text{and} \quad \pos(Y_o)=\Rn.
\end{equation}
\end{proposition}
\begin{proof}
Observe first that since $q_{\psi,Y}=\sup_{y\in Y} \ell_{\psi,y}$ is the supremum of (possibly nonproper) affine functions, it is always lower semicontinuous and convex. By Lemma~\ref{le:vector_dom}, we have $q_{\psi,Y}\not\equiv -\infty$ if and only if $Y\cap \dom(\Leg \psi)$ is nonempty. In this case, the function $q_{\psi,Y}$ is also proper since $q_{\psi,Y}\leq \psi$, which trivially follows from the definition of $\ell_{\psi,Y}$. Finally, by Lemma~\ref{le:conj_coercive}, Lemma~\ref{le:vector_dom}, and Lemma~\ref{le:pos_conv_origin}, the function $q_{\psi,Y}$ is coercive if and only if some subset of $Y\cap \dom(\Leg \psi)$ positively spans $\Rn$.
\end{proof}

\begin{lemma}
\label{le:existence_consequences}
Let $\psi\in\Convc$ and let $Y\subseteq\Rn$. If there exists $Y_o\subseteq Y$ such that
\[
Y_o\subseteq \interior(r(\psi) B^n)\quad \text{and} \quad \pos(Y_o)=\Rn,
\]
then the functions
\begin{itemize}
    \item $q_{\psi,Y}$,
    \item $q_{\psi\circ \vartheta^{-1},Y}$ for every $\vartheta\in\On$,
    \item $T_{\mathbf{\Theta}_m}(\psi)$ for every $\mathbf{\Theta}_m=\{\vartheta_1,\ldots,\vartheta_m\}\subseteq\SOn$, $m\in\N$,
    \item $q_{\psi_\rot,Y}$,
\end{itemize}
are well-defined and elements of $\Convc$. In particular, there exists a finite set $Y$ that has these properties.
\end{lemma}
\begin{proof}
It follows from the definition of $r(\psi)$ in \eqref{eq:def_r_psi} that $\interior(r(\psi) B_n)\subseteq \dom (\Leg (\psi\circ \vartheta^{-1}))$ for every $\vartheta\in \On$. Since
\[
\Leg(T_{\mathbf{\Theta}_m}(\psi))(x)=\frac{1}{m} \sum_{i=1}^m \Leg \psi(\vartheta_i^{-1} x)
\]
for $x\in\Rn$, this furthermore shows that $\interior(r(\psi) B_n)\subseteq \dom(\Leg(T_{\mathbf{\Theta}_m}(\psi)))$ for every $\mathbf{\Theta}_m=\{\vartheta_1,\ldots,\vartheta_m\}\subseteq\SOn$. Similarly, since $\psi_\rot$ is defined as the convex conjugate of \eqref{eq:rot_mean_conj}, we have $\interior(r(\psi) B_n)\subseteq \dom(\Leg(\psi_\rot))$. Thus, if $Y_o\subseteq Y$ is as in the assertion, then Proposition~\ref{prop:C_intersection} shows that all of the considered functions are well-defined and are elements of $\Convc$. Furthermore, it is easy to see that such a set $Y_o$ exists. For example, one may take $Y_o$ to be the set of vertices of a simplex $T$ such that $o\in\interior(T)$ and such that $T\subseteq \interior(r(\psi)B_n)$.
\end{proof}

\begin{remark}
\label{re:conditions_necessary}
For a given $\psi\in\Convc$, it is easy to see that if $Y\subseteq \Rn$ is such that $q_{\psi,Y}\in\Convc$, then $q_{\psi_\rot,Y}$ is not necessarily proper. For example, let $\psi\coloneq h_{[-1,1]^n}\in\Convc$, in which case we have $\Leg \psi = \ind_{[-1,1]^n}$ and, using Lemma~\ref{le:conj_rot_epi-symm}, $\Leg \psi_{\rot} = \ind_{B_n}$. If we choose $Y$ to be the set of vertices of $[-1,1]^n$ (a set that positively spans $\Rn$), then $Y\subseteq \dom(\Leg \psi)$ but $Y\cap \dom(\Leg \psi_{\rot})=\emptyset$. Thus, Proposition~\ref{prop:C_intersection} shows that $q_{\psi,Y}\in\Convc$, while it follows from Lemma~\ref{le:vector_dom} that $q_{\psi_\rot,Y}\equiv -\infty$.

Conversely, if $Y\subseteq \Rn$ is such that $q_{\psi_\rot,Y}\in\Convc$, then it is also not guaranteed that $q_{\psi,Y}$ is proper. We demonstrate this with the following example, where, for simplicity, we assume $n=2$.  Let $\psi\in \ConvcT$ be such that
\[
\Leg \psi(x)=\frac{1}{(x_1+1)^{\frac 14}}+\ind_{B_2}(x)
\]
for $x=(x_1,x_2)\in\R^2$. Clearly, $\dom(\Leg \psi)=B_2\setminus\{(-1,0)\}$. On the other hand, for $|x|=1$ we use Lemma~\ref{le:conj_rot_epi-symm} together with polar coordinates to obtain
\begin{align*}
\Leg \psi_{\rot}(x)&=\frac{1}{2\pi} \int_0^{2\pi} \frac{1}{(1+\cos(t))^{\frac 14}} \,\mathrm{d} t.
\end{align*}
Since at $t=\pi$ we have
\[
1+\cos(t) \simeq \frac 12 (t-\pi)^2 + \mathcal{O}((t-\pi)^4),
\]
the integral above converges and, therefore, $B_2\subseteq \dom(\Leg \psi_\rot)$. Since, trivially, $\dom(\Leg \psi_\rot)\subseteq \dom(\ind_{B_2})=B_2$, this shows that
\[
\dom(\Leg \psi_\rot)=B_2.
\]
Now choose
\[
Y=\{(-1,0),(1/\sqrt{2},1/\sqrt{2}),(1/\sqrt{2},-1/\sqrt{2})\}.
\]
Clearly, $\pos(Y)=\R^2$. Furthermore, $Y\subseteq B_2= \dom(\Leg \psi_{\rot})$. Thus, Proposition~\ref{prop:C_intersection} shows that $q_{\psi_{\rot},Y}\in \Convc$. However, only two points of $Y$ are also in the domain of $\Leg \psi$. Since two points are not enough to positively span $\R^2$, this means that $q_{\psi,Y}$ is not an element of $\Convc$. Furthermore, considering an appropriate rotation epi-mean of $\psi$ with a total of three suitable rotations results in a function $\varphi$ such that $Y\cap \dom(\Leg \varphi)=\emptyset$ and thus, $q_{\varphi,Y}\equiv-\infty$.\dssymb
\end{remark}


\section{Topological properties}
\label{seq:topo}
In order to establish Theorem~\ref{Schneider-main-result-A}, Schneider proved in \cite[Lemma 1]{Schneider-1967} that $K\mapsto \Phi^*(K,\mathfrak{U})$ is upper semicontinuous. The purpose of this section is to establish an analogous result for convex functions. In the course of this, we will show that for finite slope sets $Y\subseteq\Rn$, building generalized outer linearizations is continuous with respect to epi-convergence. Furthermore, we will introduce local families of functions playing the role of Euclidean unit balls ``around'' a given convex function.

\medskip

For the proof of Lemma~\ref{le:p_psi_contin} below, let us recall Carath\'eodory's theorem (see, for example, \cite[Theorem 1.1.4]{SchneiderBook}) which states that if $x\in \conv (A)$ for some $A\subseteq \Rn$, then $x$ can be written as a convex combination of $(n+1)$ or fewer points of $A$. We will also use the following consequence of Lemma~\ref{le:pos_conv_origin} and Steinitz's theorem; see \cite[Theorem 1.3.10]{SchneiderBook}, and \cite[Section 1.3, Note 4]{SchneiderBook}.

\begin{lemma}
	\label{le:pos_basis}
	If $B\subseteq\Rn$ is such that $\pos(B)=\Rn$, then there exists some subset $B_o\subseteq B$ with at most $2n$ points such that also $\pos(B_o)=\Rn$.
\end{lemma}

\begin{lemma}
	\label{le:p_psi_contin}
	Let $\psi\in\Convc$, $\vartheta\in\On$, and let $Y\subseteq\Rn$ be such that $\pos(Y)=\Rn$ and $Y\subseteq \interior(\dom(\Leg (\psi\circ \vartheta^{-1})))$. If $\psi_j\in\Convc$, $j\in\N$, is an epi-convergent sequence with limit $\psi$, then there exists $j_o\in\N$ such that $q_{\psi_j\circ\vartheta^{-1},Y}\in\Convc$ for every $j\geq j_o$. In addition,
	\begin{equation}
		\label{eq:dom_leg_q_psi_j_U_subseteq}
		\dom(\Leg q_{\psi_j\circ\vartheta^{-1},Y})\subseteq \dom(\Leg q_{\psi\circ\vartheta^{-1},Y})
	\end{equation}
	for every $j\geq j_o$ and
	\begin{equation}
		\label{eq:limsup_leg_q_psi_j_U}
		\limsup\nolimits_{j\to \infty} \Leg q_{\psi_j\circ\vartheta^{-1},Y}(x)\leq \Leg q_{\psi\circ\vartheta^{-1},Y}(x)
	\end{equation}
	for every $x\in\interior(\dom(\Leg q_{\psi\circ\vartheta^{-1},Y}))$. Furthermore, if $Y$ is finite, then also
	\begin{equation}
		\label{eq:lim_leq_psi_j_U}
		\lim\nolimits_{j\to\infty}\Leg q_{\psi_j\circ\vartheta^{-1},Y}(x)=\Leg q_{\psi\circ\vartheta^{-1},Y}(x)
	\end{equation}
	for every $x\in\Rn$ and, therefore, $q_{\psi_j\circ\vartheta^{-1},Y}$ epi-converges to $q_{\psi\circ\vartheta^{-1},Y}$ as $j\to\infty$.
\end{lemma}
\begin{proof}
	Without loss of generality, we may assume that $\vartheta=\operatorname{Id}$. Since $\pos(Y)=\Rn$, it follows from Lemma~\ref{le:pos_basis} that there exists a finite subset $Y_o\subseteq Y$ such that $\pos(Y_o)=\Rn$. Since $Y\subseteq \interior(\dom(\Leg \psi))$, the set $K_o:=\conv(Y_o)$ is a compact subset of $\interior(\dom(\Leg \psi))$, and by Lemma~\ref{le:pos_conv_origin} it has nonempty interior. By Lemma~\ref{le:epi_conv_pointwise} and Lemma~\ref{le:conj_continuous}, the sequence $\Leg \psi_j$ converges uniformly to $\Leg \psi$ on $K_o$ and, in particular, there exists $j_o\in\N$ such that $K_o\subseteq \dom(\Leg \psi_j)$ for every $j\geq j_o$. It now follows from Proposition~\ref{prop:C_intersection}  that $q_{\psi_j,Y}\in \Convc$ for every $j\geq j_o$.
	
	Next, by Lemma~\ref{le:vector_dom} we have
	\begin{equation}
		\label{eq:epi_conj_q_psi_U}
		\epi(\Leg q_{\psi,Y}) = \operatorname{cl}\big(\conv\left\{(y,t)\in\Rn\times\R:\, y\in Y\text{ s.t. } y\in\dom(\Leg\psi), \,\Leg \psi(y)\leq t \right\}\big)
	\end{equation}
	and, similarly,
	\begin{equation}
		\label{eq:epi_conj_q_psi_j_U}
		\epi(\Leg q_{\psi_j,Y}) = \operatorname{cl}\big(\conv\left\{(y,t)\in\Rn\times\R:\, y\in Y\text{ s.t. } y\in\dom(\Leg\psi_j), \,\Leg \psi_j(y)\leq t \right\}\big)
	\end{equation}
	for all $j\geq j_o$. Since $\Leg \psi(y)<\infty$ for every $y\in Y$, the above immediately implies \eqref{eq:dom_leg_q_psi_j_U_subseteq}.
	
	In order to show \eqref{eq:limsup_leg_q_psi_j_U}, we fix an arbitrary $\bar{x}\in\interior(\dom(\Leg q_{\psi,Y}))$ and let $\varepsilon>0$. Set $\bar{t}:=\Leg q_{\psi,Y}(\bar{x})$ and observe that $(\bar{x},\bar{t}+\varepsilon)\in\interior( \epi(\Leg q_{\psi,Y}))$. Thus, by \eqref{eq:epi_conj_q_psi_U} and Carath\'eodory's theorem, there exists a finite number of points $(y_1,t_1),\ldots,(y_m,t_m)$ with $y_i\in Y$ and $\Leg \psi(y_i)\leq t_i$, $i\in\{1,\ldots,m\}$, $m\in\N$, and coefficients $\lambda_1,\ldots,\lambda_m\geq 0$ with $\sum_{i=1}^m \lambda_i=1$, such that
	\[
	(\bar{x},\bar{t}+\varepsilon)=\sum_{i=1}^m \lambda_i (y_i,t_i).
	\]
	In particular, this shows that
	\[
	\Leg q_{\psi,U}(\bar{x})+\varepsilon=\sum_{i=1}^m \lambda_i t_i.
	\]
	Similar to the first part of the proof, we set $K\coloneq\conv\{y_1,\ldots,y_m\}$ and observe that this is a compact subset of $\interior(\dom(\Leg \psi))$ and that $\Leg \psi_j$ converges uniformly to $\Leg \psi$ on $K$. Thus, there exists $j_1\in\N$ such that
	\[
	|\Leg \psi(x)-\Leg \psi_j(x)|<\varepsilon
	\]
    for every $x\in K$ and $j\geq j_1$. This shows that $\Leg \psi_j(y_i)< t_i + \varepsilon$ for every $i\in\{1,\ldots,m\}$ and $j\geq j_1$. By \eqref{eq:epi_conj_q_psi_j_U}, we therefore have
    \[
    \sum_{i=1}^m \lambda_i(y_i,t_i+\varepsilon) \in \epi(\Leg q_{\psi_j,Y}).
    \]
    Since
    \[
    \sum_{i=1}^m \lambda_i(y_i,t_i+\varepsilon)= (\bar{x},\bar{t}+2\varepsilon),
    \]
    it follows that
    \[
    \Leg q_{\psi_j,Y}(\bar{x})\leq \bar{t}+2\varepsilon=\Leg q_{\psi,Y}(\bar{x})+2\varepsilon
    \]
    for every $j\geq j_1$. Since $\varepsilon>0$ was arbitrary, this proves \eqref{eq:limsup_leg_q_psi_j_U}.
	
	It remains to show \eqref{eq:lim_leq_psi_j_U}. If $Y$ is finite, then  $\conv(Y)$ is a compact subset of $\interior(\dom(\Leg \psi))$ with nonempty interior and thus, by Lemma~\ref{le:epi_conv_pointwise} and Lemma~\ref{le:conj_continuous}, the sequence $\Leg \psi_j$ converges uniformly to $\Leg \psi$ on this set. Together with \eqref{eq:epi_conj_q_psi_U} and \eqref{eq:epi_conj_q_psi_j_U}, this proves \eqref{eq:lim_leq_psi_j_U}. Another application of Lemma~\ref{le:epi_conv_pointwise} and Lemma~\ref{le:conj_continuous} now shows that $q_{\psi_j,Y}$ epi-converges to $q_{\psi,Y}$ as $j\to\infty$.
\end{proof}

While we did not show continuity of $\psi\mapsto q_{\psi\circ\vartheta^{-1},Y}$ in general, Schneider's proof in \cite{Schneider-1967} of the lower semicontinuity of $K\mapsto \Phi^*(K,\mathscr{U})$ indicates that it is enough to show continuity for certain sequences. His proof exploits the following two elementary properties of Hausdorff convergence of convex bodies:
\begin{enumerate}
    \item $K+r B_n$ converges to $K$ (in a ``nice'' way) as $r\to 0^+$.
    \item For every $r>0$ and convergent sequence of convex bodies $K_j$ with limit $K\in \Kn$, there exists $j_o\in\N$ such that $K_j\subseteq K+r B_n$ for every $j\geq j_o$.
\end{enumerate}
This suggests looking for functions $b_r\in \Convc$, $r>0$, playing the role of $r B_n$, such that functional analogues of the above properties hold true, with Hausdorff convergence of convex bodies replaced by epi-convergence of convex functions, and using infimal convolution instead of Minkowski addition. Unfortunately, in contrast to the smaller space of super-coercive convex functions, it turns out that such a universal family of functional unit balls does not exist when working on $\Convc$; see Remark~\ref{re:unit_ball} below. It is, however, possible to find functions $b_r(\psi)$, $r>0$, depending on $\psi\in\Convc$, that play the role of unit balls around $\psi$ in the sense described above.

\medskip

For a given $\psi\in\Convc$ and $r>0$, let $b_r(\psi)$ be defined by
\[
b_r(\psi)\coloneq\Leg\big(\ind_{C_r(\psi)}+r\big),
\]
where
\begin{align*}
C_r(\psi)&\coloneq \tfrac 1r B_n \cap \Big\{x\in\dom (\Leg \psi) :\, d(x,\partial \dom (\Leg \psi))\geq \min\big\{r,\tfrac{r(\psi)}{2}\big\}\Big\}\\
&=\tfrac 1r B_n \cap \Big( \operatorname{cl}(\dom(\Leg \psi)) \div \min\big\{r,\tfrac{r(\psi)}{2}\big\} B_n \Big).
\end{align*}
Here, $A\div B=\bigcap_{b\in B}(A-b)$ denotes the \textit{Minkowski difference} of $A,B\subseteq \Rn$ (see, for example, \cite[Section 3.1]{SchneiderBook}), and $d(x,C)=\inf\{\|x-y\|_2:\,y\in C\}$ denotes the distance of the point $x\in\Rn$ to the closed set $C\subseteq\Rn$. In addition, we use the convention $C_r(\psi)=\tfrac 1r B_n$ in case $r(\psi)=\infty$ or, equivalently, $\dom(\Leg \psi)=\Rn$. Note that the convexity of $\dom(\Leg \psi)$ implies that also $C_r(\psi)$ is convex. Furthermore, observe that
\[
\tfrac 1r B_n \cap \tfrac{r(\psi)}{2} B_n \subseteq C_r(\psi)
\]
for every $r>0$. Since $r(\psi)>0$ (see Lemma~\ref{le:conj_coercive} and \eqref{eq:def_r_psi}), the set $C_r(\psi)$ is therefore always nonempty and contains the origin in its interior. Hence $o\in\interior(\dom(\Leg b_r(\psi)))$, which together with Lemma~\ref{le:conj_coercive} implies that $b_r(\psi)\in \Convc$ for every $r>0$.

In analogy to the role of Euclidean balls in the definition of the Hausdorff metric, we will use the functions $b_r(\psi)$ to describe epi-convergence. The family of functions $b_r(\psi)$, $r>0$, can also be interpreted as follows. By \cite[Theorem 13.3]{RockafellarBook}, the functions $b_r(\psi)$ approximate the \textit{recession function} of $\psi$ as $r\to 0^+$, which is the smallest function $\varphi$ such that
\[
\psi \Box \varphi = \psi,
\]
see \cite[Corollary 8.5.1]{RockafellarBook}. From this point of view, the next result is natural.

\begin{lemma}
\label{le:br_infconv_seq}
For $\psi\in\Convc$, the sequence $\psi \infconv b_r(\psi)$ epi-converges to $\psi$ as $r\to 0^+$.
\end{lemma}
\begin{proof}
By Lemma~\ref{le:infconv_conj}, we have $\Leg(\psi \infconv b_r(\psi))=\Leg \psi + \Leg b_r(\psi)$, and thus
\[
\Leg(\psi \infconv b_r(\psi))(x) = \begin{cases}
\Leg\psi(x)+r\quad &\text{if } x\in C_r(\psi)\\
\infty \quad &\text{else.}
\end{cases}
\]
Note that $C_r(\psi)\subseteq \interior(\dom(\Leg \psi))$ for every $r>0$, and for every $x_0\in\interior(\dom(\Leg \psi))$ there exists $r_0> 0$ such that $x_0\in C_r(\psi)$ for every $r\leq r_0$. Therefore, $\Leg(\psi \infconv b_r(\psi))$ converges pointwise to $\Leg\psi$ on $\Rn\setminus \partial \dom(\Leg \psi)$. The result now follows by Lemma~\ref{le:epi_conv_pointwise} together with Lemma~\ref{le:conj_continuous}.
\end{proof}

\begin{lemma}
\label{le:br_infconv_ineq}
If $\psi_j\colon \Rn\to(-\infty,\infty]$, $j\in\N$, is a sequence of convex functions that epi-converges to $\psi\in\Convc$, then for every $r>0$ there exists $j_0\in\N$ such that
\[
\psi_j \geq \psi \infconv b_r(\psi)
\]
for every $j\geq j_0$.
\end{lemma}
\begin{proof}
By Lemma~\ref{le:infconv_conj}, we need to show that $\Leg \psi_j \leq \Leg \psi + \Leg b_r(\psi)$ for every $j\geq j_0$, which is equivalent to
\[
\Leg \psi_j \leq \Leg \psi + r
\]
on $C_r(\psi)$. By Lemma~\ref{le:epi_conv_pointwise} and Lemma~\ref{le:conj_continuous} together with the definition of $C_r(\psi)$, the sequence $\Leg \psi_j$ converges uniformly to $\Leg \psi$ on $C_r(\psi)$, which concludes the proof.
\end{proof}

\begin{remark}
\label{re:unit_ball}
Observe that it is not possible to find a family of functions $\bar{b}_r \in\Convc$ with $r>0$ that is independent of $\psi$, and that can be used in both Lemma~\ref{le:br_infconv_seq} and Lemma~\ref{le:br_infconv_ineq} instead of $b_r(\psi)$. In order to see this, let $\lambda >0$ be arbitrary and let $\psi=h_{\lambda B_n}$. Now for $j\in\N$, let $\psi_j = h_{(1-\frac 1j)\lambda B_n}$, which is an epi-convergent sequence with limit $\psi$. For a given $r>0$, it follows from Lemma~\ref{le:infconv_conj} that the function $\bar{b}_r$ would need to satisfy
$$\Leg \psi_j \leq \Leg \psi + \Leg \bar{b}_r$$
for every $j$ greater than or equal to some $j_0\in\N$. Since $\Leg h_K = \ind_K$ for every convex body $K$, this means that we must have $\Leg \bar{b}_r \equiv \infty$ on $\lambda B_n \setminus (1-\frac 1j) \lambda B_n$ for every $j\geq j_0$. In particular, by convexity, this implies $\Leg \bar{b}_r(x)=\infty$ whenever $|x|\geq \lambda$. Since $\lambda>0$ was arbitrary, we conclude
$$\Leg \bar{b}_r \equiv \infty$$
on $\Rn\setminus\{o\}$. Clearly, such a function cannot be used in the context of Lemma~\ref{le:br_infconv_seq}, and, in particular, by Lemma~\ref{le:conj_coercive}, it would no longer be coercive.

Note that if $\psi$ is super-coercive (or equivalently, $\dom(\Leg \psi)=\Rn$), then $b_r(\psi)=h_{\frac 1r B_n}-r$ for $r>0$, and in particular this function is independent of $\psi$. Indeed, as was shown in \cite[Lemma 5.13]{Mussnig-Li}, this function can be used in Lemma~\ref{le:br_infconv_seq}.\dssymb
\end{remark}

\begin{lemma}
\label{le:p_psi_infconv_epi-conv}
Let $\psi\in\Convc$, $\vartheta\in\On$, and let $Y\subseteq\Rn$ be such that $\pos(Y)=\Rn$ and $Y\subseteq \interior(\dom(\Leg (\psi\circ \vartheta^{-1})))$. There exists $r_o>0$ such that $q_{(\psi \infconv b_r(\psi))\circ\vartheta^{-1},Y}\in \Convc$ for every $r\leq r_o$. In addition, $q_{(\psi \infconv b_r(\psi))\circ\vartheta^{-1},Y}$ epi-converges to $q_{\psi\circ\vartheta^{-1},Y}$ as $r\to 0^+$.
\end{lemma}
\begin{proof}
Without loss of generality, we may assume that $\vartheta=\operatorname{Id}$. It immediately follows from Lemma~\ref{le:p_psi_contin} and Lemma~\ref{le:br_infconv_seq} that $q_{\psi \infconv b_r(\psi),Y}\in \Convc$ and $\dom(\Leg q_{\psi \infconv b_r(\psi),Y})\subseteq \dom(\Leg q_{\psi,Y})$
for every $r\leq r_o$ with some $r_o>0$. Furthermore,
\[
\limsup_{r\to 0^+} \Leg q_{\psi \infconv b_r(\psi),Y}(x) \leq \Leg q_{\psi,Y}(x)
\]
for every $x\in\interior(\dom(\Leg q_{\psi,Y}))$. Therefore, by Lemma~\ref{le:epi_conv_pointwise} and Lemma~\ref{le:vector_dom}, it remains to show that
\begin{equation}
\label{eq:liminf_q_psi_infconv}
\liminf_{r\to 0^+} \Leg q_{\psi \infconv b_r(\psi),Y}(x) \geq \Leg q_{\psi,Y}(x).
\end{equation}
By the definition of $b_r(\psi)$, Lemma~\ref{le:infconv_conj}, and  Lemma~\ref{le:vector_dom} we have
\begin{equation*}
	\epi(\Leg q_{\psi \infconv b_r(\psi),Y})= \operatorname{cl}\big(\conv\left\{(y,t+r)\in\R^n\times\R:\, y\in Y,\, y\in C_r(\psi),\, \Leg \psi(y)\leq t \right\}\big).
\end{equation*}
Since
\begin{equation*}
	\epi(\Leg q_{\psi,Y}) = \operatorname{cl}\big(\conv\left\{(y,t)\in\R^n\times\R:\, y\in Y, \,y\in\dom(\Leg\psi),\,\Leg \psi(y)\leq t \right\}\big)
\end{equation*}
and $C_r(\psi)\subseteq \dom(\Leg \psi)$, it follows that
\[
\Leg q_{\psi,Y}(x)<\Leg q_{\psi,Y}(x)+r\leq \Leg q_{\psi \infconv b_r(\psi),Y}(x)
\]
for every $x\in\Rn$ and $r\leq r_o$. In particular, this implies \eqref{eq:liminf_q_psi_infconv}.
\end{proof}

\begin{remark}
Looking at the proof of Lemma~\ref{le:p_psi_infconv_epi-conv} and considering that $C_r(\psi)$ approximates $\dom (\Leg \psi)$ from the inside, one might be tempted to think that for every $x\in\interior(\dom \Leg \psi)$, the equality
\[
\Leg q_{\psi,Y}(x)+r = \Leg q_{\psi \infconv b_r(\psi),Y}(x)
\]
is attained whenever $r>0$ is sufficiently small. However, in general, this is not true, as the following example in $\R^2$ shows. Let $\psi\in \ConvcT$ be such that
\[
\Leg \psi(x)=x_1^2
\]
for $x=(x_1,x_2)\in\R^2$, and consider the set  $Y\subseteq\R^2$ defined by
\[
Y\coloneq\{(0,0)\}\cup\{(1,x_2) :\, x_2\in\R\}\cup \{(-1,x_2) : \, x_2\in\R\}.
\]
It follows that
\[
\Leg q_{\psi,Y}(x)=|x_1|+\mathbf{I}_{[-1,1]}(x_1),
\]
and, therefore, the line $\{(0,x_2):\,x_2\in\R\}$ is contained in the interior of $\dom(\Leg q_{\psi,Y})$. Now observe that for $x_2\neq 0$, the points $((0,x_2),0)$, which are in the boundary of $\epi (\Leg q_{\psi,Y})$, are not in the closure of the convex combination of any number of points of the form $(y,t)$ with $y\in C_r(\psi)$ and $\Leg \psi(y)\leq t$. In particular, for these points, we have
\[
\Leg q_{\psi,Y}((0,x_2)) + r < \Leg q_{\psi\infconv b_r(\psi),Y}((0,x_2))
\]
for every $r>0$.\dssymb
\end{remark}

Following the ideas presented in \cite[Lemma 1]{Schneider-1967}, together with the new tools introduced in this section, we can now establish the upper semicontinuity of $\psi\mapsto \Phi^*(\psi,\mathcal{Y})$, recalling its definition in \eqref{mainDef}. For this, we need to formally introduce the following convention. In Proposition~\ref{upper-semi-lem} we consider expressions of the form $\Phi^*(\psi_j,\mathcal{Y})$, where we can a priori not guarantee that $\mathcal{Y}$ is in $\Ycoll_{\psi_j}$. In particular, it may happen that $q_{\psi_j\circ \vartheta^{-1},Y}$ is no longer an element of $\Convc$, in which case we set $\Phi(q_{\psi_j\circ \vartheta^{-1},Y})=\infty$.

\begin{proposition}
\label{upper-semi-lem}
Let $\psi\in \Convc$ and $\mathcal{Y} \subseteq \Ycoll_\psi$. If $\psi_j\in \Convc$, $j\in\N$, epi-converges to $\psi$ as $j\to\infty$, then for every $\vartheta\in\On$ and $Y\in \mathcal{Y}$ there exists $j_o=j_o(Y,\vartheta)\in\N$ such that $q_{\psi_j\circ\vartheta^{-1},Y}\in\Convc$ for every $j\geq j_o$. In addition, if $\Phi\colon \Convc\to[-\infty,\infty]$ is upper semicontinuous, monotone decreasing, and concave with respect to infimal convolution, then
\begin{equation}
\label{eq:limsup_phi_star}
\limsup\nolimits_{j\to\infty} \Phi^*(\psi_j,\mathcal{Y})\leq \Phi^*(\psi,\mathcal{Y}).
\end{equation}
In other words, $\Phi^*(\cdot,\mathcal{Y})$ is upper semicontinuous at $\psi$.
\end{proposition}
\begin{proof}
Throughout the proof, let $\psi\in \Convc$ and $\mathcal{Y}\subseteq\Ycoll_\psi$ be given. For $Y\in\mathcal{Y}$ and $\vartheta\in\On$, it follows from the definitions of $\Ycoll_\psi$ and $r(\psi)$ that there exists $Y_o\subseteq Y$ such that $\pos(Y_o)=\Rn$ and $Y_o\subseteq \interior(\dom(\Leg(\psi\circ\vartheta^{-1})))$. Thus, Lemma~\ref{le:p_psi_contin} shows that there exists $j_o\in\N$ such that the functions $q_{\psi_j\circ\vartheta^{-1},Y_o}$ are elements of $\Convc$ for every $j\geq j_o$. It now follows from Lemma~\ref{q-is-monotone-wrt-inclusion} that also
\[
q_{\psi_j\circ\vartheta^{-1},Y}\in\Convc
\]
for every $j\geq j_o$.

Now fix $\varepsilon>0$ and assume, without loss of generality, that $\Phi^*(\psi,\mathcal{Y})<\infty$. By the definition of $\Phi^*(\psi,\mathcal{Y})$, there exist $\vartheta\in\SOn$ (or $\vartheta\in\Oo$ if $n=1$, respectively) and $Y\in\mathcal{Y}$ such that
\[
\Phi(q_{\psi\circ\vartheta^{-1},Y}) < \Phi^*(\psi,\mathcal{Y})+\varepsilon.
\]
By Lemma~\ref{le:p_psi_infconv_epi-conv} and the upper semicontinuity of $\Phi$, there exists $r>0$ such that
\[
\Phi(q_{(\psi\infconv b_r(\psi))\circ\vartheta^{-1},Y}) \leq \Phi(q_{\psi\circ\vartheta^{-1},Y})+\varepsilon.
\]
Moreover, by Lemma~\ref{le:br_infconv_ineq}, there exists $j_1\in\N$ such that $\psi_j \geq \psi \infconv b_r(\psi)$ for every $j\geq j_1$, which trivially implies that
\[
q_{\psi_j\circ\vartheta^{-1},Y} \geq q_{(\psi\infconv b_r(\psi))\circ\vartheta^{-1},Y}
\]
for every $j\geq j_1$. Together with the monotonicity of $\Phi$, we obtain
\[
\Phi(q_{\psi_j\circ\vartheta^{-1},Y}) \leq \Phi(q_{(\psi\infconv b_r(\psi))\circ\vartheta^{-1},Y}) \leq \Phi(q_{\psi\circ\vartheta^{-1},Y})+\varepsilon < \Phi^*(\psi,\mathcal{Y})+2\varepsilon.
\]
Thus, by the definition of $\Phi^*$,
\[
\Phi^*(\psi_j,\mathcal{Y})\leq \Phi^*(\psi,\mathcal{Y})+2\varepsilon
\]
for every $j\geq j_1$. This proves \eqref{eq:limsup_phi_star}, which means that $\Phi^*(\cdot,\mathcal{Y})$ is upper semicontinuous at $\psi$.
\end{proof}

We close this section with a stability result for outer linearizations under scaling of a finite set of slopes. We shall later make use of this in Section~\ref{suse:approx}.

First, let us recall an elementary fact about convex functions. While convex functions are generally continuous only on the interiors of their domains, lower semicontinuous convex functions are continuous relative to every line segment in their domain. See, for example, \cite[Theorem 2.35]{Rockafellar-Wets}.

\begin{lemma}
\label{le:conv_onedim_cont}
Let $\varphi\colon\Rn\to{(-\infty,\infty]}$ be lower semicontinuous and convex. If $I\subseteq \dom(\varphi)$ is a line segment, then the restriction $\varphi\vert_I$ is continuous.
\end{lemma}

\begin{lemma}
\label{le:interior-slopes-dense}
Let $\psi\in\Convc$ and let $Y=\{y_1,\ldots,y_N\}\subseteq\dom(\Leg\psi)$ be a finite set such that $\pos(Y)=\Rn$. For every $t\in(0,1)$, we have
\[
tY=\{ty:\,y\in Y\}\subseteq \interior(\dom(\Leg\psi))
\qquad\text{and}\qquad
\pos(tY)=\Rn,
\]
and therefore $Y\in\Ycoll_\psi$. In addition, $q_{\psi,tY}$ epi-converges to $q_{\psi,Y}$ as $t\to 1^-$. 
\end{lemma}
\begin{proof}
Since $\psi\in\Convc$, it follows from Lemma~\ref{le:conj_coercive} that $o\in \interior(\dom(\Leg\psi))$. Hence there exists $r>0$ such that $rB_n\subseteq \dom(\Leg\psi)$. Now fix an arbitrary $y\in Y\subseteq \dom(\Leg\psi)$ and let $t\in(0,1)$. Since $\dom(\Leg\psi)$ is convex, we have
\[
ty+(1-t)z \in \dom(\Leg \psi)
\]
for every $z\in rB_n$. 
Therefore,
\[
ty+(1-t)rB_n\subseteq \dom(\Leg\psi),
\]
which shows that $ty\in \interior(\dom(\Leg\psi))$. Since $y\in Y$ was arbitrary, this proves that $tY\subset \interior(\dom(\Leg\psi))$. Moreover, since $t>0$ and positive scalar multiplication does not change the positive hull, we have $\pos(tY)=\pos(Y)=\Rn$. 

It remains to show that $q_{\psi,tY}$ epi-converges to $q_{\psi,Y}$ as $t\to 1^-$. For $i\in\{1,\ldots,N\}$ and $t\in(0,1]$, let
\[
\ell_{i,t}(x)=\langle x, ty_i\rangle - \Leg \psi(t y_i)
\]
for $x\in\Rn$. By \eqref{eq:def_outer_lin} and Lemma~\ref{le:vector_dom}, we have
\[
q_{\psi,tY}(x) = \max\nolimits_{i\in\{1,\ldots,N\}} \ell_{i,t}(x)\qquad \text{and} \qquad q_{\psi,Y}(x)= \max\nolimits_{i\in\{1,\ldots,N\}} \ell_{i,1}(x),
\]
and thus
\begin{align*}
|q_{\psi,tY}(x)-q_{\psi,Y}(x)|&\leq \max\nolimits_{i\in\{1,\ldots,N\}} |\ell_{i,t}(x)-\ell_{i,1}(x)|\\
&\leq (1-t) |x| \max\nolimits_{i\in\{1,\ldots,N\}} |y_i| + \max\nolimits_{i\in\{1,\ldots,N\}} |\Leg \psi(ty_i)-\Leg \psi(y_i)|
\end{align*}
for every $x\in\Rn$. Now for every $i\in\{1,\ldots,N\}$ it follows from Lemma~\ref{le:conv_onedim_cont} that $\Leg \psi$ is continuous on $[0,y_i]\subseteq \dom(\Leg \psi)$. Hence, for every $x\in\Rn$, the above expression goes to zero as $t\to 1^-$. We have thus shown pointwise convergence of $q_{\psi,tY}$ to $q_{\psi,Y}$, which by Lemma~\ref{le:epi_conv_pointwise} implies epi-convergence.
\end{proof}


\section{Proofs of Theorem~\ref{mainThmB} and Theorem~\ref{mainThmA}}\label{sec:proofs-main-thms}

The main step in the proof of Theorem~\ref{mainThmA} is showing that $\Phi^*(\cdot,\mathcal{Y})$ is monotone with respect to rotation epi-means.

\begin{lemma}\label{mainLem}
Let $n\geq 2$, $\psi\in \Convc$, $\mathcal{Y}\subseteq \Ycoll_\psi$, and let $\Phi\colon \Convc\to[-\infty,\infty]$ be monotone decreasing and concave with respect to infimal convolution. For every finite set of rotations $\mathbf{\Theta}_m\subseteq \SOn$, we have $\Phi^*(\psi,\mathcal{Y})\leq \Phi^*(T_{\mathbf{\Theta}_m}(\psi), \mathcal{Y})$.
For $n=1$, the same inequality is true if we consider $\mathbf{\Theta}_m\subseteq \Oo$ instead.
\end{lemma}

\begin{proof}
Assume first that $n\geq 2$. Let $\varepsilon > 0$ and $\mathbf{\Theta}_m\subseteq \SOn$ be given and assume, without loss of generality, that $\Phi^*(T_{\mathbf{\Theta}_m}(\psi), \mathcal{Y})<\infty$. By the definition of $\Phi^*$ in \eqref{mainDef}, there exist $\vartheta_o \in \SOn$ and $Y_o\in \mathcal{Y}$ such that
\begin{equation}
\label{eq:epsilon-1}
\Phi\big(q_{T_{\mathbf{\Theta}_m}(\psi)\circ\vartheta_o^{-1},Y_o}\big) < \Phi^* (T_{\mathbf{\Theta}_m}(\psi),\mathcal{Y}) + \varepsilon.
\end{equation}
For each $i\in\{1,\ldots,m\}$, set $p_i\coloneq q_{\psi\circ(\vartheta_o\,\vartheta_i)^{-1},Y_o}$, where 
$\vartheta_1,\ldots,\vartheta_m$ are the rotations belonging to $\mathbf{\Theta}_m$, and set $P\coloneq \frac{1}{m}\sq\bigsquare_{i=1}^m p_i$. By \eqref{eq:def_outer_lin} and Lemma \ref{le:vector_dom}, together with Lemma~\ref{le:infconv_conj}, for every $x\in\Rn$ we obtain
\begin{align*}
q_{P,Y_o}(x)&=\sup_{y\in Y_o}\left\{\langle x,y\rangle-\Leg P(y)\right\}\\
&=\sup_{y\in Y_o}\left\{\langle x,y\rangle-\frac{1}{m}\sum_{i=1}^m\Leg p_i(y)\right\}\\
&=\sup_{y\in Y_o}\left\{\langle x,y\rangle-\frac{1}{m}\sum_{i=1}^m\Leg(\psi\circ(\vartheta_o\,\vartheta_i)^{-1})(y)\right\}\\
&=\sup_{y\in Y_o}\left\{\langle x,y\rangle-\Leg T_{\mathbf{\Theta}_m(\psi)\circ\vartheta_o^{-1}}(y)\right\}\\
&=q_{T_{\mathbf{\Theta}_m}(\psi)\circ\vartheta^{-1},Y_o}(x).
\end{align*}
Hence, by \eqref{Y-equals-Rn} and Lemma~\ref{q-is-monotone-wrt-inclusion},
\begin{equation}\label{mainLem-ineq}
\frac{1}{m}\sq\bigsquare_{i=1}^m p_i = P = q_{P,\Rn} \geq q_{P,Y_o} = q_{T_{\mathbf{\Theta}_m}(\psi)\circ\vartheta^{-1},Y_o}.
\end{equation}
Thus, since $\Phi$ is monotone decreasing and concave, and since trivially $\Phi(p_i) \geq \Phi^*(\psi,\mathcal{Y})$ for every $i\in\{1,\ldots,m\}$, we obtain
\[
\Phi\big(q_{T_{\mathbf{\Theta}_m}(\psi)\circ\vartheta^{-1},Y_o}\big) \geq \Phi\left(\frac{1}{m} \sq \bigsquare_{i=1}^m p_i\right) \geq\frac{1}{m}\sum_{i=1}^m \Phi(p_i)
\geq \frac{1}{m}\sum_{i=1}^m \Phi^*(\psi,\mathcal{Y})=\Phi^*(\psi,\mathcal{Y}).
\]
Together with \eqref{eq:epsilon-1}, we thus have
\[
\Phi^*(\psi,\mathcal{Y}) < \Phi^* (T_{\mathbf{\Theta}_m}(\psi),\mathcal{Y}) + \varepsilon.
\]
Since $\varepsilon>0$ was arbitrary, the desired statement follows.

Finally, in the case $n=1$, we can carry out the proof in the same way as above by substituting $\Oo$ where appropriate and keeping in mind that the definition of $\Phi^*$ in the one-dimensional case also uses $\Oo$ in the infimum.
\end{proof}


\subsection{Proof of Theorem~\ref{mainThmA}}
Since $\mathcal{Y}\subseteq \Ycoll_\psi$, for every $Y\in\mathcal{Y}$ there exists $Y_o\subseteq Y$ such that $\pos(Y_o)=\Rn$ and $Y_o\subseteq \interior(r(\psi) B_n)$. Thus, it follows from Lemma~\ref{le:existence_consequences} that $q_{\psi\circ \vartheta^{-1},Y},q_{\psi_\rot,Y}\in\Convc$ for every $Y\in\mathcal{Y}$ and $\vartheta\in\On$, ensuring that $\Phi^*(\psi,\mathcal{Y})$ and $\Phi^*(\psi_\rot,\mathcal{Y})$ are well-defined.

For $n=1$, the desired inequality directly follows from \eqref{eq:def_rot_onedim} together with Lemma~\ref{mainLem}, since $\psi_\rot=T_{\mathbf{\Theta}_m}(\psi)$ with $\mathbf{\Theta}_m=\{1,-1\}=\Oo$.

For $n\geq 2$, Lemma~\ref{epi-symm-convergence} shows that there exists a sequence $T_{\mathbf{\Theta}_m}(\psi)$, $m\in\N$, of rotation epi-means of $\psi$ such that $T_{\mathbf{\Theta}_m}(\psi)$ epi-converges to $\psi_\rot$ as $m\to\infty$. Therefore, by Proposition~\ref{upper-semi-lem} and Lemma~\ref{mainLem}, we obtain
\[
\Phi^*(\psi,\mathcal{Y})\leq \limsup\nolimits_{m\to\infty}\Phi^*(T_{\mathbf{\Theta}_m}(\psi),\mathcal{Y})\leq \Phi^*\left(\epilim\nolimits_{m\to\infty} T_{\mathbf{\Theta}_m}(\psi),\mathcal{Y}\right)=\Phi^*(\psi_\rot,\mathcal{Y}).
\]
\qed

\begin{remark}
\label{re:semicontinuity}
A closer look at the proof of Theorem~\ref{mainThmA} reveals that the functional $\Phi$ need not be upper semicontinuous everywhere on $\Convc$. Rather,  upper semicontinuity is only required along $q_{(\psi_\rot\infconv b_r(\psi_\rot))\circ \vartheta^{-1},Y}$ as $r\to 0^+$, with $\vartheta\in\SOn$ and $Y\in\mathcal{Y}\subseteq \Ycoll_\psi$ (cf.\ the proof of Proposition~\ref{upper-semi-lem}, which is the only place where the upper semicontinuity of $\Phi$ is used). We shall make use of this observation in Section~\ref{se:covering}.
\dssymb
\end{remark}


\subsection{Proof of Theorem \ref{mainThmB}}
Since $\Psi$ is convex and invariant under rotations, we have 
\[
\Psi(T_{\mathbf{\Theta}_m}(\psi)) = \Psi \left(\frac{1}{m} \sq \bigsquare_{i=1}^m(\psi\circ\vartheta_i^{-1})\right)
\leq \frac{1}{m} \sum_{i=1}^m \Psi(\psi\circ\vartheta_i^{-1})
= \Psi(\psi)
\]
for every rotation epi-mean $T_{\mathbf{\Theta}_m}(\psi)$ of $\psi$. For $n=1$ and assuming that $\Psi$ is invariant under reflections, this already gives the desired statement since $T_{\{1,-1\}}(\psi)=\psi_\rot$.

For $n\geq 2$, as in the proof of Theorem \ref{mainThmA}, let $T_{\mathbf{\Theta}_m}(\psi)$, $m\in\N$, be such that  $T_{\mathbf{\Theta}_m}(\psi) \stackrel{\epi}{\longrightarrow} \psi_{\rot}$ as $m\to\infty$. 
Together with the lower semicontinuity of $\Psi$, this implies
\[
   \Psi(\psi_{\rot}) = \Psi\left(\epilim\nolimits_{m\to\infty} T_{\mathbf{\Theta}_m}(\psi)\right) \leq \liminf\nolimits_{m\to\infty} \Psi(T_{\mathbf{\Theta}_m}(\psi)) \leq \Psi(\psi).
\]
\qed


\subsection{Retrieving Theorem \ref{Schneider-main-result-A} from Theorem \ref{mainThmA}}\label{geometric-special-case}

For a convex body $K\in\Kn$, let $\psi_K:=\ind_K$ be the convex indicator function of $K$, which is an element of $\Convc$. Note  that $\Leg\psi_K=h_K$ is the support function of $K$, and $\dom(\Leg\psi_K)=\Rn$. For any slope set $Y\subseteq\Rn$, it directly follows from Lemma~\ref{le:vector_dom} below that 
\[
q_{\psi_K,Y}(x)=\sup\nolimits_{y\in Y}\left(\langle x,y\rangle-h_K(y)\right)\leq \sup\nolimits_{y\in\Rn}\left(\langle x,y \rangle - h_K(y)\right) = \Leg h_K(x)=\psi_K(x)
\]
for $x\in\Rn$. The generalized outer linearization $q_{\psi_K,Y}$
determines the outer representation
\[
P_{K,Y}\coloneq\bigcap_{y\in Y}\left\{x:\,\langle x,y\rangle\leq h_K(y)\right\},
\]
which can be equivalently described as $P_{K,Y}=\{x :\, q_{\psi_K,Y}\leq 0\}$, i.e., the 0-sublevel set of $q_{\psi_K,Y}$. Note, that trivially $K\subseteq P_{K,Y}$.
If $Y$ is finite, then $P_{K,Y}$ is a polyhedron which circumscribes $K$. Moreover, $\pos(Y)=\Rn$ if and only if $P_{K,Y}$ is bounded (cf.\ Proposition~\ref{prop:C_intersection}), matching the spanning hypothesis in Schneider's Theorem \ref{Schneider-main-result-A}.

More specifically, if for $U\in \mathscr{U}_n$ we set $Y_U=\{\lambda u : \, u\in U, \lambda \geq 0\}$, then
\begin{align*}
    q_{\psi_K,Y_U}(x) &= \sup\nolimits_{u\in U, \lambda \geq 0} \left( \langle x,\lambda u\rangle - h_K(\lambda u) \right)\\
    &=\begin{cases}
        0\quad &\text{if } \langle x,u\rangle \leq h_K(u) \; \forall u\in U,\\
        \infty\quad &\text{else},
    \end{cases}\\
    &=\ind_{P(K,U)}.
\end{align*}
Now let $\Phi$ be a functional on $\Convc$ satisfying the hypotheses of Theorem \ref{mainThmA}. This induces a functional $\Phi_{\K}$ on $\Kn$ via $\Phi_{\K}(K)=\Phi(\ind_K)$ for $K\in\Kn$.
The preceding remarks show that Theorem~\ref{Schneider-main-result-A} for the functional $\Phi_K$ can be retrieved from Theorem~\ref{mainThmA}. In fact, the proof of Theorem~\ref{mainThmA} reduces, after the above identifications, to the same structural steps as in Schneider's original proof of Theorem~\ref{Schneider-main-result-A} in \cite{Schneider-1967}: admissible slope sets correspond to admissible outer normal sets; the functional combination $\lambda\sq\psi_K\square(1-\lambda)\sq\psi_L$ reduces to the corresponding Minkowski combination of $K,L\in\Kn$; and rotational epi-symmetrization corresponds to rotational averaging (i.e., Hadwiger's rotation means and the mean width rearrangement $K_{\rm rot}$ of $K$). It should be noted that in this particular case, all occurring convex functions are super-coercive, which significantly simplifies the proofs presented here.


\section{Log-concave functions}\label{sec:log-concave}

A function $f\colon \Rn\to[0,\infty)$ is \emph{log-concave} if $\log f$ is concave. Every log-concave function can be expressed as $f=e^{-\psi}$ for some convex function $\psi\colon\Rn\to\R$. Let 
\[
\LCc\coloneq\{f=e^{-\psi}:\,\psi\in\Convc\}.
\]
For $f,g\in\LCc$, the \emph{Asplund sum} (or \emph{supremal convolution}) $f\star g$ is the function defined by
\[
(f\star g)(x) \coloneq \sup\nolimits_{y\in\Rn}f(y)g(x-y),\quad x\in\Rn,
\]
provided the supremum exists. For $\lambda>0$, the \emph{$\lambda$-homothety} $\lambda\cdot f$ is defined by
\[
(\lambda\cdot f)(x)\coloneq f\left(\frac{x}{\lambda}\right)^\lambda,\quad x\in\Rn.
\]
The set $\LCc$ is closed under these operations, i.e., for all $f,g\in\LCc$ and all $\lambda>0$, we have $f\star g\in\LCc$ and $\lambda\cdot f\in\LCc$ (see, e.g., \cite[Lemma 2.3]{Hofstatter-Schuster}). The Asplund sum of $f=e^{-\psi}, g=e^{-\varphi}\in\LCc$ is related to the infimal convolution of $\psi,\varphi\in\Convc$ via the relation
\begin{equation}\label{inf-sup-convolutions}
    (a\cdot f\star b\cdot g)(x) = e^{-(a\sq\psi\square b\sq\varphi)(x)},\quad x\in\Rn,
\end{equation}
for all $a,b>0$.

Let $F$ be a functional on $\LCc$, where we do not allow $F$ to attain both $0$ and $\infty$. We say that $F$ is \emph{monotone increasing} if $F(f)\leq F(g)$ whenever $f\leq g$. We also say that $F$ is \emph{invariant under rotations} if $F(f\circ\vartheta)=F(f)$ for every $f\in\LCc$ and every $\vartheta\in\SOn$. When we say that $F$ is upper (or lower) semicontinuous, we mean with respect to the topology of hypo-convergence. Here, a sequence of functions $f_j=e^{-\psi_j}\in\LCc$, $j\in\N$, \emph{hypo-converges} to $f=e^{-\psi}\in\LCc$ if for every $x\in\Rn$:
\begin{itemize}
    \item $f(x)\geq \limsup_{j\to\infty}f_j(x_j)$ for every sequence $x_j\in\Rn$, $j\in\N$, such that $x_j\to x$, and
    \item $f(x)=\lim_{j\to\infty}f_j(x_j)$ for some sequence $x_j\in\Rn$, $j\in\N$, such that $x_j\to x$.
\end{itemize}
Equivalently, $\psi_j$ epi-converges to $\psi$. Under the additional assumption that $f$ has full-dimensional support, this is in turn equivalent to pointwise convergence almost everywhere.

We say that $F\colon\LCc\to[0,\infty]$ is \emph{log-concave} (with respect to the Asplund sum) if for all $f,g\in\LCc$ and every $\lambda\in(0,1)$
\begin{equation}\label{log-concave-F}
    F(\lambda\cdot f\star(1-\lambda)\cdot g)\geq F(f)^\lambda F(g)^{1-\lambda}.
\end{equation}
Here, we do not allow $F$ to attain both $0$ and $\infty$, but only one of these values respectively. If the inequality in \eqref{log-concave-F} is reversed, then we say that $F$ is \emph{log-convex}. In particular, the Pr\'ekopa--Leindler inequality states that the \emph{total mass functional} $J\colon\LCc\to [0,\infty)$, given by
\[
J(f)\coloneq\int_{\Rn} f(x) \dint x,\quad f\in\LCc,
\]
satisfies \eqref{log-concave-F}, with equality if and only if $f$ and $g$ are translates (see, e.g., \cite[Theorem 9.5.1]{SchneiderBook}).
For more background on log-concave functions, we refer the reader to, e.g., \cite{Colesanti-inbook} and \cite[Section 9.5]{SchneiderBook}.

\medskip

Next, we formulate the extremal problem we will study in the log-concave setting.

\begin{definition}\label{mainDef-2}
     Let $F\colon\LCc\to [0,\infty]$. Given a function $f=e^{-\psi}\in\LCc$ and a set $\mathcal{Y}\subseteq\Ycoll_\psi$, we define 
\begin{equation}\label{mainDef-logconcave} 
F^*(f, \mathcal{Y}) \coloneq \inf\{ F(e^{-q_{\psi\circ\vartheta^{-1},Y}}) :\,\vartheta\in \SOn,\, Y\in \mathcal{Y}\}.
\end{equation}
\end{definition}

The analogue of the ball $K_\rot$ for log-concave functions is the rotational hypo-sym-\linebreak metrization, defined below.

\begin{definition}\label{mainDef-3}
    Let $f\in\LCc$. The \emph{rotational hypo-symmetrization} of $f$ is defined by
    \begin{equation}
    \label{eq:rot_mean_conj-lc}
    f_{\rot}(x)\coloneq e^{-\psi_{\rot}(x)}
    \end{equation}
    for $x\in\Rn$.
\end{definition}
\noindent Note that by Lemma \ref{le:conj_rot_epi-symm},  the rotational hypo-symmetrization of $f$ always exists and is unique.

\medskip

Considering the correspondences between operations on $\LCc$ and $\Convc$ described above, we obtain the following consequence of Theorem~\ref{mainThmA}.

\begin{theorem}
\label{mainThmA-logconcave}
If $F\colon \LCc\to{[0,\infty]}$ is upper semicontinuous, log-concave, and monotone increasing, then
\[
F^*(f,\mathcal{Y}) \leq F^*(f_{\rot},\mathcal{Y})
\]
for every $f=e^{-\psi}\in\LCc$ and $\mathcal{Y}\subseteq \Ycoll_\psi$. In particular, the above expressions are well-defined.
\end{theorem}
\begin{proof}
Observe first that, similar to the proof of Theorem~\ref{mainThmA}, Lemma~\ref{le:existence_consequences} shows that $F^*(f,\mathcal{Y})$ and $F^*(f_\rot,\mathcal{Y})$ are well-defined.

In case $F^*(f,\mathcal{Y})=0$, the statement is trivial. In the remaining case $F^*(f,\mathcal{Y})>0$, we consider the map $\Phi\colon \Convc\to[-\infty,\infty]$ given by
\[
\Phi(\psi)\coloneq
\begin{cases}
\log F(e^{-\psi}) \quad &\text{if }F(e^{-\psi})>0\\
-\infty\quad  &\text{if }F(e^{-\psi})=0
\end{cases}
\]
for $\psi\in\Convc$. We will show that this functional satisfies the hypotheses of Theorem~\ref{mainThmA}.

Suppose that the sequence $\{\psi_j\}\subseteq\Convc$, $j\in\N$, epi-converges to $\psi\in\Convc$. Then the corresponding sequence $f_j\coloneq e^{-\psi_j}\in\LCc$ hypo-converges to $f\coloneq e^{-\psi}\in\LCc$. By the upper semicontinuity of $F$, we have
\[
\limsup\nolimits_{j\to\infty}F(f_j)\leq F(f).
\]
We claim that this implies
\[
\limsup\nolimits_{j\to\infty}\Phi(\psi_j)\leq \Phi(\psi),
\]
where $\Phi$ is allowed to take the value $-\infty$. Indeed, if $F(f)=0$, then $\Phi(\psi)=-\infty$, and the claim follows because $\limsup_{j}F(f_j)\le 0$ together with the nonnegativity of $F$ implies $F(f_j)\to 0$, hence $\Phi(\psi_j)\to -\infty$ along every subsequence with $F(f_j)>0$, while $\Phi(\psi_j)=-\infty$ whenever $F(f_j)=0$. If $F(f)>0$, then for every subsequence $f_{j_k}$, $k\in\N$, we have either $F(f_{j_k})=0$ for infinitely many $k$, in which case $\Phi(\psi_{j_k})=-\infty$ along that further subsequence and there is nothing to prove, or else $F(f_{j_k})>0$ eventually. In the latter case, using the continuity and monotonicity of $\log$ on $(0,\infty)$, we get
\[
\limsup\nolimits_{k\to\infty}\Phi(\psi_{j_k})
=\limsup\nolimits_{k\to\infty}\log F(f_{j_k})
\leq \log\!\left(\limsup\nolimits_{k\to\infty}F(f_{j_k})\right)
\leq \log F(f)
=\Phi(\psi).
\]
Since this holds for every subsequence, $\Phi$ is upper semicontinuous on $\Convc$. 

Next, let $\psi_1,\psi_2\in\Convc$, $\lambda\in(0,1)$, and let $f_1\coloneq e^{-\psi_1}, f_2\coloneq e^{-\psi_2}\in\LCc$.
Since $F$ is log-concave,
\[
F\left( e^{-(\lambda\sq\psi_1\square(1-\lambda)\sq\psi_2)}\right) = F(\lambda\cdot f_1\star(1-\lambda)\cdot f_2)\geq F(f_1)^\lambda F(f_2)^{1-\lambda}.
\]
Taking logarithms and using the fact that the logarithm is an increasing function, from the previous inequality, we obtain
\begin{align*}
\Phi(\lambda\sq\psi_1\square(1-\lambda)\sq\psi_2) &\geq \log\left[F(f_1)^\lambda F(f_2)^{1-\lambda}\right]\\
&=\lambda\log F(f_1)+(1-\lambda)\log F(f_2)=\lambda\Phi(\psi_1)+(1-\lambda)\Phi(\psi_2).
\end{align*}
Therefore, $\Phi$ is concave on $\Convc$. 

Lastly, if $\psi_1,\psi_2\in\Convc$ are such that $\psi_1\leq\psi_2$, then $e^{-\psi_1}\geq e^{-\psi_2}$, so by the monotonicity of $F$ we get $F(e^{-\psi_1}) \geq F(e^{-\psi_2})$. Since the logarithm is an increasing function, this implies
\[
\Phi(\psi_1) = \log F(e^{-\psi_1}) \geq \log F(e^{-\psi_2}) = \Phi(\psi_2),
\]
which shows that $\Phi$ is monotone. Thus, $\Phi$ satisfies the hypotheses of Theorem \ref{mainThmA}.

Finally, note that by the definition of $F^*$ and $\Phi^*$, for $n\geq 2$ we have
\begin{align*}
    \Phi^*(\psi,\mathcal{Y}) &=\inf\left\{\Phi(q_{\psi\circ\vartheta^{-1},Y}):\,\vartheta\in{\rm SO}(n), Y\in\mathcal{Y}\right\}\\
    &=\inf\left\{\log F(e^{-q_{\psi\circ\vartheta^{-1},Y}}):\,\vartheta\in{\rm SO}(n), Y\in\mathcal{Y}\right\}\\
    &=\log\left(\inf\left\{F(e^{-q_{\psi\circ\vartheta^{-1},Y}}):\,\vartheta\in{\rm SO}(n), Y\in\mathcal{Y}\right\}\right)\\
    &=\log F^*(f,\mathcal{Y}),
\end{align*}
and in the case $n=1$ we need to consider $\Oo$ in the infima above. To interchange the infimum and the logarithm, we used the fact that $\log$ is increasing (and $F^*(f,\mathcal{Y})>0$). Hence,
\[
F^*(f,\mathcal{Y})=\exp(\Phi^*(\psi,\mathcal{Y}))\quad \text{and}\quad F^*(f_{\rot},\mathcal{Y})=\exp(\Phi^*(\psi_{\rot},\mathcal{Y})).
\]
Applying Theorem~\ref{mainThmA} to $\Phi$  yields $\Phi^*(\psi,\mathcal{Y})\leq \Phi^*(\psi_{\rm rot},\mathcal{Y})$, and exponentiating we get $F^*(f,\mathcal{Y})\leq F^*(f_{\rm rot},\mathcal{Y})$. This completes the proof.
\end{proof}

The analogue of Theorem~\ref{mainThmB} for log-concave functions reads as follows.

\begin{theorem}
\label{mainThmB-logconcave}
Let $n\geq 2$. If $G\colon\LCc\to[0,\infty]$ is a lower semicontinuous, log-convex functional on $\LCc$ that is invariant under rotations, then
\[
G(f_\rot) \leq G(f) 
\]
for every $f\in\LCc$.  For $n=1$, the same inequality is true if we assume $G$ to be log-convex on $\LCc$ and invariant under reflections.
\end{theorem}
\begin{proof}
Similar to the proof of Theorem~\ref{mainThmA-logconcave}, we consider $\psi\mapsto \log(G(e^{-\psi}))$, which defines a lower semicontinuous, convex functional on $\Convc$ that is invariant under rotations. The statement now follows from Theorem~\ref{mainThmB}.
\end{proof}

\begin{remark}
    The framework and extremal results developed here for coercive log-concave functions extend naturally to the broader class of $\alpha$-concave functions where $\alpha\in[-\infty,\infty]$. In this setting, $\alpha$-concavity is formulated in terms of the $\alpha$-means $M_\alpha^{(s,t)}(u,v)$ and $\alpha$-Asplund operations; see \cite{Rotem2013,Salani-2015}.  The proofs of the corresponding results for $\alpha$-concave functions follow along the same lines as those for log-concave functions, with the obvious modifications, and some additional bookkeeping to keep track of the parameter $\alpha$. Thus, to keep the exposition focused, and since the adaptations are straightforward, we do not include the $\alpha$-concave development here.\dssymb
\end{remark}


\section{Applications}\label{sec:applications}

In this section, we present some applications of our main results.

\subsection{A functional Urysohn-type inequality}

We establish the following functional version of Urysohn's inequality. We remark that this result can also be obtained from the techniques presented in \cite[Section 6]{Salani-2015} by Salani, who studied solutions to elliptic PDEs in convex sets. Further functional variants of Urysohn's inequality were obtained, for example, in \cite{BCF-2014,Hoehner-2023,Milman-Rotem,Rotem2012,Rotem2013}.

\begin{theorem}\label{urysohn-2}
    For every $f\in{\rm LC}_{\rm c}(\R^n)$ we have $J(f)\leq J(f_{\rm rot})$.
\end{theorem}

For the proof of the above inequality, we need the following result, which is essentially a consequence of the dominated convergence theorem. See, for example, \cite[Lemma 16]{CLM-IMRN}, where a more general result was shown. See also \cite[Lemma 3.2]{Mussnig-Li}.

\begin{lemma}\label{J-cont-lem}
The total mass $J\colon\LCc\to[0,\infty)$ is continuous.
\end{lemma}

\begin{proof}[Proof of Theorem \ref{urysohn-2}]
    Let $F=J$ denote the total mass functional, $\mathcal{Y}=\{\Rn\}$, and fix $f=e^{-\psi}\in\LCc$. By Lemma \ref{J-cont-lem}, $F$ is continuous with respect to hypo-convergence, and it is log-concave by the Pr\'ekopa--Leindler inequality. Moreover, it is monotone. Hence, the hypotheses of Theorem \ref{mainThmA-logconcave} are satisfied. Note that by Lemma \ref{le:q_psi_rn_psi},
    \begin{equation*}
\psi\circ\vartheta^{-1} = q_{\psi\circ\vartheta^{-1},\Rn}
    \end{equation*}
    for every $\vartheta\in\On$, which, together with the rotation invariance of the total mass functional (or reflection invariance, in case $n=1$), shows that
    \begin{equation*}
    J(\exp(-q_{\psi,\Rn}))=F^*(f,\mathcal{Y}) \leq F^*(f_{\rot},\mathcal{Y}) = J(f_{\rot}).
    \end{equation*}
    Therefore, we obtain 
    \[
    J(f)=J(e^{-\psi})=J(\exp(-q_{\psi,\Rn}))\leq J(f_{\rot}).
    \]
\end{proof}

\begin{remark}
    Choosing
    \[
        f(x)=\chi_K(x)\coloneq\begin{cases}
            1\quad &\text{if } x\in K,\\
            0\quad &\text{else},
        \end{cases}
    \]    
    for some $K\in\Kn$ in Theorem \ref{urysohn-2}, we recover the classical Urysohn inequality \eqref{urysohn-ineq}:
    \[
\vol_n(K)=J(\chi_K)\leq J((\chi_K)_{\rot})=J(\chi_{K_{\rot}})=\vol_n(K_{\rot}).
    \]
    Here we used $(\chi_K)_{\rot}=\chi_{K_{\rot}}$.\dssymb
\end{remark}


\subsection{Approximation of functions by outer linearizations}
\label{suse:approx}
Throughout this section, we look at outer linearizations of convex functions (and subsequently log-concave functions) with a finite number of slopes. Given $\psi\in\Convc$ and $Y=\{y_1,\ldots,y_N\}$, $N\in\N$, we consider the piecewise affine function $q_{\psi,Y}$, which is called an \emph{outer linearization} of $\psi$. To ensure that the slopes of $Y$ actually make a contribution to $q_{\psi,Y}$, we will assume, in light of Lemma~\ref{le:vector_dom}, that
\[
Y\subseteq \dom(\Leg \psi).
\]
In addition, we will be interested in the case when $q_{\psi,Y}\in\Convc$. By Proposition~\ref{prop:C_intersection}, this is the case if and only if
\[
\pos(Y)=\Rn,
\]
which is possible only if
\[
N\geq n+1.
\]
Minimal sets $Y$ are, for example, given by the vertices of a simplex in $\dom(\Leg \psi)$ that contains the origin in its interior. For applications of outer linearizations to convex optimization and stochastic programming, we refer the reader to \cite{Bertsekas2011, BL-stochastic-book}.

Let us note that the situation here is somewhat more general than in the rest of this article, in that we do not necessarily require $Y$ to be a subset of the interior of $\dom(\Leg \psi)$. On the other hand, however, we only consider finite sets $Y$ in this subsection.

\medskip

Next, given $f=e^{-\psi}\in\LCc$, let $q_{\psi,Y}\in\Convc$ be an outer linearization of its convex base function $\psi\in\Convc$. We call the function $e^{-q_{\psi,Y}}$ an \emph{outer log-linearization} of $f$. An example is illustrated in Figure~\ref{fig:outer_log_lin}.

\begin{figure}[!ht]
\begin{center}
 \begin{tikzpicture}[scale=1.5]
\begin{axis}[
width=0.32\textwidth,
    axis lines = left,
    xlabel = {$x$},
    ylabel = {$\psi(x)=x^2$},
    ymin=-2, ymax=10,
    xtick={-2,-1,0,1,2},
    ytick={-2,0,2,4,6,8},
     x label style={at={(axis description cs:1,0)},anchor=west},
    y label style={at={(axis description cs:0,1)},rotate=270,anchor=south},
    label style={font=\tiny},
                    tick label style={font=\tiny},
]
\addplot [
    name path=parabola,
    domain=-3:3, 
    samples=100, 
    color=black,
]
{x^2};
\addplot[mark=*, mark size=0.75pt] coordinates {(-2,4)};
\addplot[mark=*, mark size=0.75pt] coordinates {(-1/2,1/4)};
\addplot[mark=*, mark size=0.75pt] coordinates {(3/4,9/16)};
\addplot[mark=*, mark size=0.75pt] coordinates {(3/2,9/4)};

        \addplot[name path=bluecurve,
        color=blue, 
        mark=none] coordinates {(-3,8)(-5/4,1)(1/4,-1/2)(5/4,3/2)(3,27/4)}; 

\end{axis}
\end{tikzpicture}
 \begin{tikzpicture}[scale=1.5]
\begin{axis}[
width=0.32\textwidth,
    axis lines = left, 
    xlabel = {$x$},
    ylabel = {$f(x)=e^{-x^2}$},
    ymin=0, ymax=2,
    xtick={-2,-1,0,1,2},
    ytick={0,0.5,1,1.5},
     x label style={at={(axis description cs:1,0)},anchor=west},
    y label style={at={(axis description cs:0,1)},rotate=270,anchor=south},
    label style={font=\tiny},
                    tick label style={font=\tiny},
]
\addplot [
    domain=-3:3, 
    samples=100, 
    color=black,
]
{exp(-x^2)};

\addplot[mark=*, mark size=.75pt] coordinates {(-2,0.01831563888)};
\addplot[mark=*, mark size=.75pt] coordinates {(-1/2,0.77880078307)};
\addplot[mark=*, mark size=.75pt] coordinates {(3/4,0.56978282473)};
\addplot[mark=*, mark size=.75pt] coordinates {(3/2,0.10539922456)};



\addplot [name path = arc1,
    domain=-3:-5/4, 
    samples=100, 
    color=red,
]
{exp(4*x+4)};

        \addplot [name path=arc2,
    domain=-5/4:1/4, 
    samples=100, 
    color=red,
]
{exp(x+1/4)};

\addplot [name path=arc3,
    domain=1/4:5/4, 
    samples=100, 
    color=red,
]
{exp(-2*x+1)};
\addplot [name path=arc4,
    domain=5/4:3, 
    samples=100, 
    color=red,
]
{exp(-3*x+9/4)};
\end{axis}
\end{tikzpicture}
\end{center}
\caption{\label{fig:outer_log_lin}On the left, an outer linearization of $\psi(x)=x^2$ with four slopes is shown in blue. On the right, the corresponding outer log-linearization of $f(x)=e^{-x^2}$ is shown in red.}
\end{figure}
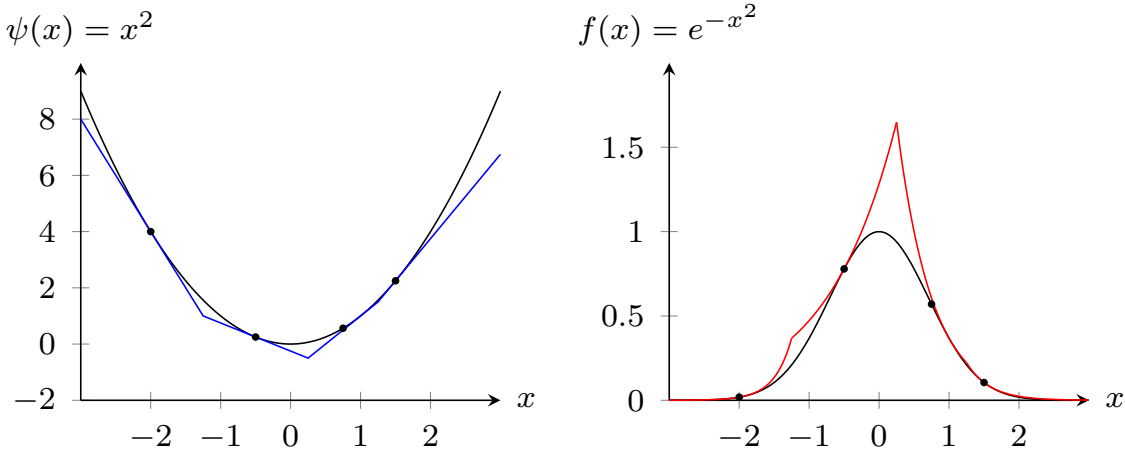

For a given function $f=e^{-\psi}\in\LCc$, we want to consider the set of all integrable outer-log linearizations of $f$, generated from at most $N\geq n+1$ slopes. Thus, let
\[
\mathscr{P}_N(f)\coloneq\{e^{-q_{\psi,Y}} :\, Y\subseteq \dom(\Leg \psi),\, \pos(Y)=\Rn,\, |Y|\leq N\}.
\]
We will investigate the infimum of the total mass functional $J\colon\LCc\to{[0,\infty)}$ on this set, and therefore we define
\[
F_N(f)\coloneq\inf\nolimits_{p\in \mathscr{P}_N(f)} J(p).
\]
Observe that since trivially $f\leq p$ pointwise for every $p\in \mathscr{P}_N(f)$ and since the total mass functional is monotone increasing, we always have $J(f)\leq J(p)$ for every $p\in \mathscr{P}_N(f)$, and, therefore,
\[
J(f)\leq F_N(f).
\]
The quantity $F_N(f)$ is therefore an approximation of $J(f)$ through outer-log linearizations of $f$.

\begin{remark}
    The set of outer log-linearizations of a log-concave function $f$ with at most $N$ bounding log-affine hyperplanes may be regarded as the conceptual ``dual'' of the set of inner log-linearizations with at most $N$ break points in the hypograph of $f$. The latter construction was previously studied in \cite{Hoehner-2023,Hoehner-Novaes,PB-2020,Rinott}.\dssymb
\end{remark}

As an application of Theorem~\ref{mainThmA-logconcave}, we prove the following Urysohn-type inequality for the outer approximation of coercive log-concave functions.

\begin{theorem}
\label{thm:outer_log_urysohn}
	Let $n,N\in\N$ be such that $N\geq n+1$. For every $f\in\LCc$, we have
	\[
	F_{N}(f) \leq F_{N}(f_{\rot}).
	\]
\end{theorem}
\begin{proof}
First, we consider the case $n\geq 2$. Let $f=e^{-\psi}\in\LCc$ be fixed and let
\[
\mathcal{Y}_{\psi,N}\coloneq\{Y\subseteq \interior(r(\psi) B_n):\, |Y|\leq N,\, \pos(Y)=\Rn\}\subseteq \Ycoll_\psi.
\]
Observe that
\[
\mathcal{Y}_{\psi,N}\subseteq \{Y\subseteq \dom(\Leg \psi) :\, |Y|\leq N,\, \pos(Y)=\Rn\},
\]
and that for $Y\in\mathcal{Y}_{\psi,N}$ and $\vartheta\in\SOn$, we also have
\begin{equation*}
\vartheta Y =\{\vartheta y:\, y\in Y\} \in \mathcal{Y}_{\psi,N}.
\end{equation*}
Thus, by the rotation invariance of $J$,
\begin{align*}
F_N(f)&=\inf\{J(e^{-q_{\psi,Y}}):\,Y\subseteq \dom(\Leg \psi),\, |Y|\leq N, \,\pos(Y)=\Rn \}\\
&\leq \inf\{J(e^{-q_{\psi,Y}}):\,Y\in \mathcal{Y}_{\psi,N}\}\\
&=\inf\{J(e^{-q_{\psi\circ \vartheta^{-1},\vartheta Y}}):\,\vartheta\in\SOn, Y\in \mathcal{Y}_{\psi,N}\}\\
&=\inf\{J(e^{-q_{\psi\circ \vartheta^{-1},Y}}):\, \vartheta\in\SOn, Y\in \mathcal{Y}_{\psi,N}\}\\
&= J^*(f,\mathcal{Y}_{\psi,N}).
\end{align*}
Next, since $\interior(\dom(\Leg \psi_\rot)) = \interior(r(\psi) B_n)$, it follows from
Lemma~\ref{J-cont-lem}, Lemma~\ref{le:interior-slopes-dense}, and the radial symmetry of $f_\rot=e^{-\psi_\rot}$, that
\[
F_N(f_\rot)=\inf\{J(e^{-q_{\psi_\rot \circ \vartheta^{-1},Y}}) :\, \vartheta\in\SOn, Y\in\mathcal{Y}_{\psi,N}\}\\
=J^*(f_\rot,\mathcal{Y}_{\psi,N}).
\]
Hence, Theorem~\ref{mainThmA-logconcave} shows that
\[
F_N(f)\leq J^*(f,\mathcal{Y}_{\psi,N})\leq J^*(f_\rot,\mathcal{Y}_{\psi,N}) = F_N(f_\rot).
\]
For $n=1$, the proof proceeds in the same way as above, except that we need to use $\Oo$ where appropriate.   
\end{proof}


\subsection{A covering result}\label{se:covering}

For a probability measure $\mu$ on $\Rn$, let $\supp(\mu)$ denote its support, meaning the complement of the largest open subset on which the measure vanishes. Assuming that $\supp(\mu)$ is compact, let $R_\mu\in[0,\infty)$ be given by
\[
R_\mu\coloneq \sup\nolimits_{y\in \supp(\mu)} |y|.
\]
For $\psi\in\Convc$ such that $\supp(\mu) \subseteq \dom (\Leg \psi)$, we consider
\[
\Phi_{\rm avg}(\psi,\mu)\coloneq \int_{\Rn} \Leg\psi(y) \dint\mu(y).
\]
Note that the above condition on $\supp(\mu)$ is fulfilled, in particular, when $r(\psi)>R_\mu$ (recall the definition of $r(\psi)$ in \eqref{eq:def_r_psi}).

We have the following consequence of Theorem~\ref{mainThmA}.

\begin{corollary}
\label{cor:covering_functions}
Let $n\geq 2$ and let $\mu$ be a probability measure with compact support on $\Rn$. If $\psi\in\Convc$ is such that $r(\psi)>R_\mu$, then there exists $\vartheta_o\in\SOn$ such that
\[
\Phi_{\rm avg}(\psi \circ \vartheta_o^{-1},\mu) = \min\nolimits_{\vartheta \in \SOn} \Phi_{\rm avg}(\psi \circ \vartheta^{-1},\mu) \leq \Phi_{\rm avg}(\psi_\rot,\mu).
\]
For $n=1$, the same statement holds if we consider $\Oo$ instead of $\SOo$.
\end{corollary}
\begin{proof}
Let $\psi\in\Convc$ with $r(\psi)>R_\mu$ be fixed throughout the proof. We consider the functional $\Phi\colon\Convc\to(-\infty,\infty]$, given by
\[
\Phi(\varphi)\coloneq \Phi_{\rm avg}(\varphi,\mu)
\]
for $\varphi\in \Convc$.

It follows from Lemma~\ref{le:infconv_conj} that $\Phi$ is linear, and thus concave with respect to infimal convolution. If $\varphi_1\geq \varphi_2$ pointwise, then $\Leg \varphi_1 \leq \Leg \varphi_2$, and therefore, $\Phi(\varphi_1)\leq \Phi(\varphi_2)$, which shows that $\Phi$ is monotone decreasing. Now let $\varphi\in\Convc$ be such that $r(\varphi)>R_\mu$ and let
\[
M_{\varphi,\mu}\coloneq \max\nolimits_{y\in R_\mu B_n} |\Leg \varphi(y)|,
\]
which is finite since $\Leg \varphi$ is finite in a neighborhood of $R_\mu B_n$. Next let $\varphi_j\in\Convc$, $j\in\N$, be epi-convergent to $\varphi$ as $j\to\infty$. It follows from Lemma~\ref{le:epi_conv_pointwise} and Lemma~\ref{le:conj_continuous} that $\Leg \varphi_j$ converges uniformly to $\Leg \varphi$ on $R_\mu B_n$ and thus,
\[
|\Leg \varphi_j(y)|\leq M_{\varphi,\mu}+1
\]
for every $y\in R_\mu B_n$ and every $j\geq j_o$ with some $j_o\in\N$. Considering that $\mu$ is a probability measure, it now follows from the dominated convergence theorem that
\[
\lim\nolimits_{j\to\infty} \Phi(\varphi_j)=\lim\nolimits_{j\to\infty}\int_{\Rn} \Leg \varphi_j(y)\dint \mu(y) = \int_{\Rn} \Leg\varphi(y)\dint \mu(y) = \Phi(\varphi).
\]
In particular, since $r(\psi_\rot)=r(\psi)>R_\mu$, it follows from the definition of $b_r(\psi_\rot)$ that also $r(\psi \infconv b_r(\psi_\rot))>R_\mu$ for every sufficiently small $r$. Together with Lemma~\ref{le:q_psi_rn_psi} and Lemma~\ref{le:br_infconv_seq} this shows
\begin{align*}
\lim\nolimits_{r\to 0^+} \Phi(q_{(\psi_\rot \infconv b_r(\psi_\rot))\circ\vartheta^{-1},\Rn})&=
\lim\nolimits_{r\to 0^+} \Phi((\psi_\rot\infconv b_r(\psi_\rot))\circ\vartheta^{-1})\\
&=\Phi(\psi_\rot \circ\vartheta^{-1})\\
&=\Phi(q_{\psi_\rot \circ \vartheta^{-1},\Rn})
\end{align*}
for every $\vartheta\in\On$. Thus, choosing $\mathcal{Y}=\{\Rn\}$, it follows from Theorem~\ref{mainThmA} and Remark~\ref{re:semicontinuity} that
\begin{align}
\label{eq:inf_phi_avg_ineq}
\inf\left\{\Phi_{\rm avg}(\psi\circ \vartheta^{-1},\mu) :\, \vartheta\in \SOn \right\} 
&=\inf\left\{\Phi(q_{\psi\circ \vartheta^{-1},\Rn}) :\, \vartheta\in \SOn \right\}\\
&=\Phi^*(\psi,\{\Rn\})\notag\\
&\leq \Phi^*(\psi_\rot,\{\Rn\})\notag\\
&=\Phi_{\rm avg}(\psi_\rot,\mu),\notag
\end{align}
where we have used Lemma~\ref{le:q_psi_rn_psi}. The case $n=1$ is handled in the same way, only now we need to consider $\Oo$ in the infimum.

It remains to show that the infimum in \eqref{eq:inf_phi_avg_ineq} is attained. Note that this is trivially true for $n=1$, and we will therefore assume without loss of generality that $n\geq 2$ in the following. Since $\Leg\psi$ is continuous on $R_\mu B_n$, the map
\[
\vartheta \mapsto \Leg\psi(\vartheta^{-1} y),\quad \vartheta\in\SOn,
\]
is continuous for every $y\in R_\mu B_n$. Considering that $|\Leg\psi|$ is bounded on $R_\mu B_n$ by $M_{\psi,\mu}$, and that $\mu$ is a probability measure, it follows from the dominated convergence theorem that
\[
\vartheta \mapsto \int_{\Rn} \Leg \psi(\vartheta^{-1}y)\dint\mu(y) = \Phi_{\rm avg}(\psi\circ \vartheta^{-1},\mu)
\]
is continuous on $\SOn$. Since $\SOn$ is compact, this shows that the minimum in \eqref{eq:inf_phi_avg_ineq} is attained.
\end{proof}

We will now explain how Corollary~\ref{cor:covering_functions} serves as a covering result, assuming for simplicity that $n \geq 2$ in the following exposition. Given a probability measure $\mu$ with compact support and $\psi\in\Convc$ such that $r(\psi)>R_\mu$, let $\vartheta_o\in\SOn$ be a minimizing rotation as in the statement of Corollary~\ref{cor:covering_functions}. We now define $p_{\psi}\colon \Rn\to\R$ as
\[
p_{\psi}\coloneq q_{\psi\circ \vartheta_o^{-1},\supp(\mu)}=\sup\nolimits_{y\in \supp(\mu)} \ell_{\psi\circ \vartheta_o^{-1},y},
\]
where we note that $p_\psi$ does not attain $\pm \infty$ since $r(\psi)>R_\mu$ ensures that $y\in \dom(\Leg(\psi\circ \vartheta^{-1}))$ for every $y\in\supp(\mu)$ and $\vartheta\in\SOn$.
By Lemma~\ref{le:vector_dom}, this means that
\[
\epi(\Leg p_{\psi}) = \operatorname{cl}\big(\conv{(y,t)\in\Rn\times\R:\, y\in \supp(\mu) \text{ s.t. } \Leg \psi(\vartheta_o^{-1} y)\leq t}\big).
\]
Since $\Phi_{\rm avg}(\psi\circ \vartheta_o^{-1},\mu)=\Phi_{\rm avg}(p_{\psi},\mu)$ by construction, it follows that for a given $\psi$ and a prescribed set of admissible slopes $Y=\supp(\mu)$, there exists a rotation $\vartheta_o$ such that the outer linearization $p_\psi$ of $\psi\circ \vartheta_o^{-1}$, generated by slopes in $Y$, has average intercept data, weighted according to $\mu$, not exceeding the rotation-invariant benchmark $\Phi_{\rm avg}(\psi_\rot,\mu)$. In other words, among all rotated halfspace covers of $\psi$ formed from the prescribed slope set $Y$, there exists one whose $\mu$-weighted average intercept data is controlled by the rotation-invariant reference function $\psi_\rot$.

\begin{remark}
Of particular interest is the case when $\mu$ is a discrete measure, in which case $p_\psi$ is a polyhedral function. We will demonstrate how this yields, notably, a covering result by Firey and Groemer \cite{Firey-Groemer-1964}, as formulated by Schneider \cite{Schneider-1967}. 

Let $\mu=\sum_{i=1}^m \lambda_i \delta_{u_i}$, where $\lambda_1,\ldots,\lambda_m\geq 0$ are such that $\lambda_1+\cdots+\lambda_m=1$ and $\delta_{u_i}$ denotes the Dirac measure concentrated on $u_i\in\Sp$, $1\leq i\leq m$. Next, let $K\in\Kn$ be a convex body containing the origin and let $\psi=\ind_K$. We now have
\[
\Phi_{\rm avg}(\psi,\mu) =\sum_{i=1}^m \lambda_i h_K(u_i).
\]
Furthermore, since
\[
(\ind_K)_{\rot}=\ind_{\frac 12 w(K) B_n},
\]
it follows from Corollary~\ref{cor:covering_functions} that there exists a rotation $\vartheta_o\in\SOn$ such that
\[
\sum_{i=1}^m \lambda_i h_{\vartheta_o K}(u_i) = \sum_{i=1}^m \lambda_i h_K(\vartheta_o^{-1} u_i) \leq \sum_{i=1}^m\lambda_i h_{\frac 12 w(K) B_n}(u_i) = \frac{w(K)}{2}.
\]
Now let $P=\{x:\, p_{\psi}(x)\leq 0\}$. Then $P$ is a polyhedron whose outer facet normals belong to $\{u_1,\ldots,u_m\}$, and $h_P(u_i)=h_{\vartheta_o K}(u_i)$ for the active facets. Note that this implies 
\[
\Phi_{\rm avg}(\ind_P,\mu)=\sum_{i=1}^m \lambda_i h_P(u_i) = \sum_{i=1}^m \lambda_i h_{\vartheta_o K}(u_i)\leq \frac{w(K)}{2}.
\]
Thus, among the polyhedra with facets orthogonal to the $u_i$, we have found one, namely $P$, that contains a rotated copy of $K$ and whose ``size'', given through $\Phi_{\rm avg}(\ind_P,\mu)$, is bounded from above by $w(K)/2$.
\dssymb
\end{remark}


\section*{Acknowledgments}
The authors wish to thank Aris Daniilidis, Liran Rotem, and Jacopo Ulivelli for helpful comments and remarks. Fabian Mussnig was supported, in part, by the Austrian Science Fund (FWF): 10.55776/P36210.



\vspace{3mm}

\noindent {\sc Department of Mathematics \& Computer Science, Longwood University}

\noindent {\it E-mail address:} {\tt hoehnersd@longwood.edu}

\medskip

\noindent {\sc Mathematics Department, University of Salzburg}

\noindent {\it E-mail address:} {\tt fabian.mussnig@plus.ac.at}

\end{document}